\documentclass[a4paper, 10pt, notitlepage]{article}

\usepackage{amsthm, amsmath, amssymb, latexsym}
\usepackage{mathrsfs,cite}

\theoremstyle{plain}
\newtheorem{theorem}{Theorem}[section]	
\newtheorem{lemma}[theorem]{Lemma}	
\newtheorem{corollary}[theorem]{Corollary}
\newtheorem{proposition}[theorem]{Proposition}

\theoremstyle{definition}
\newtheorem{definition}[theorem]{Definition}
\newtheorem{example}[theorem]{Example}

\theoremstyle{remark}
\newtheorem{remark}[theorem]{Remark}

\usepackage{chngcntr}
\counterwithin{theorem}{section}	

\numberwithin{equation}{section}	

\DeclareMathOperator{\co}{co}
\DeclareMathOperator{\cone}{cone}

\DeclareMathOperator{\linhull}{span}

\DeclareMathOperator{\extreme}{ext}
\DeclareMathOperator{\cl}{cl}

\author{Dolgopolik M.V.\footnote{Institute for Problems in Mechanical Engineering, Saint Petersburg, Russia}}
\title{A New Constraint Qualification and Sharp Optimality Conditions for Nonsmooth Mathematical Programming
Problems in Terms of Quasidifferentials}

\begin{document}

\maketitle

\begin{abstract}
The paper is devoted to an analysis of a new constraint qualification and a derivation of the strongest existing
optimality conditions for nonsmooth mathematical programming problems with equality and inequality constraints in terms
of Demyanov-Rubinov-Polyakova quasi\-differentials under the minimal possible assumptions. To this end, we obtain a
novel description of convex subcones of the contingent cone to a set defined by quasidifferentiable equality and
inequality constraints with the use of a new constraint qualification. We utilize these description and constraint
qualification to derive the strongest existing optimality conditions for nonsmooth mathematical programming problems in
terms of quasidifferentials under less restrictive assumptions than in previous studies. The main feature of the new
constraint qualification and related optimality conditions is the fact that they depend on individual elements of 
quasidifferentials of the objective function and constraints and are not invariant with respect to the choise of
quasidifferentials. To illustrate the theoretical results, we present two simple examples in which optimality conditions
in terms of various subdifferentials (in fact, any outer semicontinuous/limiting subdifferential) are satisfied at a
nonoptimal point, while the optimality conditions obtained in this paper do not hold true at this point, that is,
optimality conditions in terms of quasidifferentials, unlike the ones in terms of subdifferentials, detect the 
nonoptimality of this point.
\end{abstract}

\section{Introduction}

A class of nonsmooth quasidifferentiable functions was introduced by Demyanov, Rubinov, and Polyakova in the late
1970s \cite{DemRubPol79,DemyanovRubinov80}. Since then, several collections of papers
\cite{DemyanovDixon,Quasidifferentiability_book} and monographs
\cite{DemVas_book,DemRub_book,QuasidiffMechanics} were devoted to quasidifferential calculus and its applications in
the finite dimensional case. Infinite dimensional extensions of quasidifferential calculus were analysed in
\cite{DemRub83,KusraevKutateladze,PalRecUrb,Uderzo,Dolgopolik_AbstrConvApprox,BasaevaKusraevKutateladze}. 
A generalization of the concept of quasidifferentiability called $\varepsilon$-\textit{quasidifferentiability} was
proposed by Gorokhovik \cite{Gorokhovik82,Gorokhovik84,Gorokhovik86,Gorokhovik_book}. Another generalized concept of
quasidifferentiability was introduced by Ishizuka \cite{Ishizuka}.

Necessary conditions for an unconstrained local minimum in terms of quasidifferentials were first obtained by Polyakova
\cite{Polyakova81}. In \cite{DemyanovPolyakova80}, Demyanov and Polyakova studied optimality conditions in terms of
quasidifferentials for the problem
\begin{equation} \label{InequalityConstrainedProblem}
  \min \: f_0(x) \quad \text{subject to} \quad g(x) \le 0.
\end{equation}
Note that problems with several inequality constraints $g_i(x) \le 0$ can be easily reduced to the case of a single
constraint by setting $g(x) = \max_{i} g_i(x)$. 

As is well-known (see, e.g. \cite[Example~1]{LudererRosiger90}), optimality conditions for quasidifferentiable
programming problems cannot be formulated in the traditional way involving the Lagrangian, which results in
the fact that optimality conditions for such problems can be stated in several non-equivalent forms. Optimality
conditions for problem \eqref{InequalityConstrainedProblem} from \cite{DemyanovPolyakova80} were formulated in geometric
terms and involved some cones generated by a quasidifferential of the constraint. Optimality conditions for problem
\eqref{InequalityConstrainedProblem} similar to Fritz John and KKT conditions in which Lagrange multipliers depend on
individual elements of quasidifferentials were studied in \cite{Shapiro86,LudererRosiger90,KuntzScholtes}. Fritz
John-type optimality conditions for problem \eqref{InequalityConstrainedProblem} were derived by Sutti \cite{Sutti}.
Some connections between KKT form and geometric form of optimality conditions for problem
\eqref{InequalityConstrainedProblem} were pointed out by Dinh et. al \cite{DinhTuan2003,DinLeeTuan2005}. Uderzo
\cite{Uderzo2002} obtained optimality conditions for problem \eqref{InequalityConstrainedProblem} in terms of a
quasidifferential of the nonlinear Lagrangian $L(x) = p(f_0(x), g(x))$, where $p$ is an (unknown) sublinear
function. Finally, various constraint qualifications for problem \eqref{InequalityConstrainedProblem} were discussed in
\cite{Ward91,KuntzScholtes92,KuntzScholtes}, while independence of constraint qualifications and optimality conditions
for problem \eqref{InequalityConstrainedProblem} on the choice of quasidifferentials (recall that a quasidifferential is
not uniquely defined) was analyzed in \cite{LudererRosiger91,Luderer92}.

A geometric form of optimality conditions in terms of quasidifferentials for problems with a single \textit{equality}
and no inequality constraints was obtained by Polyakova \cite{Polyakova86}. Optimality conditions from
\cite{Polyakova86} were further analyzed by Wang and Mortensen in \cite{WangMortensen}, where some results on
independence of optimality conditions on the choice of quasidifferentials were presented as well. Similar optimality
conditions for problems with constraints of the form $F(x) = 0$ or $F(x) \le 0$, where $F$ is a so-called
\textit{scalarly quasidifferentiable} mapping between infinite dimensional spaces, were derived by Glover et al.
\cite{Glover92,GloverJeyakumarOettli94} and Uderzo \cite{Uderzo2,Uderzo2007}.

Optimality conditions in terms of quasidifferentials for nonsmooth mathematical programming problems with equality,
inequality, and nonfunctional constraints were first studied by Shapiro \cite{Shapiro84,Shapiro86}. These conditions
were formulated in terms of a quasidifferential of the $\ell_1$ penalty function. Optimality conditions for nonsmooth
mathematical programming problems involving quasidifferentials of the objective function and inequality constraints, and
the Clarke subdifferentials of the equality constraints were derived by Gao \cite{Gao2000_Clarke}. KKT optimality
conditions for such problems involving the Demyanov difference of quasidifferentials were studied in
\cite{Gao2000,Gao2002,XiaSongZhang2005,SongXiaZhang2006}. However, it is very hard to compute the Demyanov difference of
a quasidifferential in nontrivial cases, which makes such conditions less appealing for applications, than optimality
conditions in terms of quasidifferentials. To the best of author's knowledge, first KKT-type optimality conditions in
terms of quasidifferentials for nonsmooth mathematical programming problems with equality and inequality constraints
were obtained in the recent paper \cite{Dolgopolik_MetricReg} with the use of a Mangasarian-Fromovitz-type constraint
qualification in terms of quasidifferentials.

Finally, the problem of when necessary optimality conditions for quasidifferentiable problems become sufficient ones was
analyzed in \cite{Glover92,Antczak2017} under generalized invexity assumptions, while optimality conditions for vector
quasidifferentiable optimization problems were studied by Glover et al. \cite{GloverJeyakumarOettli94}, 
Basaeva \cite{Basaeva2004,Basaeva2008} (see also \cite{KusraevKutateladze,BasaevaKusraevKutateladze}), and 
Antczak \cite{Antczak2016}.

The main goal of this article is to obtain a convenient description of convex subcones of the contingent cone to a set
defined by quasidifferentiable equality and inequality constraints and give a new perspective on constraint 
qualifications and optimality conditions for nonsmooth mathematical programming in terms of quasidifferentials. Unlike
all existing results, we aim at obtaining conditions that depend on individual elements of quasidifferentials and
\textit{might not be satisfied for some of them}. Such conditions provide additional flexibility that allows one to
obtain much sharper results than the use of quasidifferentials as a whole. To this end, being inspired by the papers of
Di et al. \cite{Di1,Di2} on a derivation of the classical KKT optimality conditions under weaker assumptions, we present
a completely new description of convex subcones of the contingent cone to a set defined by quasidifferentiable equality
and inequality constraints. This description leads to a new natural constraint qualification for nonsmooth mathematical
programming problems in terms of quasidifferentials that we utilize to derive the strongest existing optimality
conditions for such problems under less restrictive assumptions than in all previous studies on quasidifferentiable
programming problems. See Remark~\ref{remark:InEqConstr_Diff} and Section~\ref{Section_CQ_Comparison}
for a detailed comparison of our assumptions with the assumptions used in previous studies. See also
\cite{Dolgopolik_ConstrainedCalcVar} for applications of the main results of this paper to constrained nonsmooth
problems of the calculus of variations.

To illustrate our theoretical results, we present an example with a degenerate constraint in which all existing
constraint qualifications for quasidifferentiable programming problems fail, while our constraint qualification holds
true. Moreover, we demonstrate that in some cases optimality conditions in terms of quasidifferentials are better than
optimality conditions in terms of various subdifferentials. Namely, we give two examples in which optimality conditions
in terms of the Clarke subdifferential \cite[Theorem~6.1.1]{Clarke}, the Michel-Penot subdifferential \cite{Ioffe93},
the approximate (Ioffe) subdifferential \cite[Proposition~12]{Ioffe84}, the basic Mordukhovich subdifferential
\cite[Theorem~5.19]{Mordukhovich_II}, and the Jeyakumar-Luc subdifferential \cite[Corollary~3.4]{WangJeyakumar} (in
fact, any outer semicontinuous/limiting subdifferential; see, e.g. \cite{Ioffe2012,Penot_book}) are satisfied at a
nonoptimal point, while optimality conditions in terms of quasidifferentials do \textit{not} hold true at this point.
Thus, quasidifferential-based optimality conditions in some cases detect the nonoptimality of a given point, when
subdifferential-based conditions fail to do~so.

The paper is organized as follows. A description of convex subcones of the contingent cone to a set defined by
quasidifferentiable equality and inequality constraints, as well as related constraint qualifications, are presented in
Section~\ref{Section_ContingentCones}. In Section~\ref{Section_OptimalityConditions}, this description is utilized to
obtain the strongest existing necessary optimality conditions for nonsmooth mathematical programming problems in terms
of quasidifferential under less restrictive assumptions than in previous studies. In this section we also present two
examples demonstrating that optimality conditions in terms of quasidifferentials are sometimes better than optimality
conditions in terms of various subdifferentials. A comparison between assumptions and constraint qualifications used
in this paper and in previous studies is presented in Section~\ref{Section_CQ_Comparison}. Finally, for the sake of
completeness, some basic definitions from quasidifferential calculus are collected in
Section~\ref{Section_Preliminaries}.

\section{Quasidifferentiable Functions}
\label{Section_Preliminaries}

From this point onwards, let $X$ be a real Banach space. Its topological dual space is denoted by $X^*$, whereas the
canonical duality pairing between $X$ and $X^*$ is denoted by $\langle \cdot, \cdot \rangle$. Finally, denote by
$\cl^*$ the closure in the weak${}^*$ topology.

Let $U \subset X$ be an open set. Recall that a function $f \colon U \to \mathbb{R}$ is called directionally
differentiable (d.d.) at a point $x \in U$ iff for any $v \in X$ there exists the finite limit
$$
  f'(x, v) = \lim_{\alpha \to +0} \frac{f(x + \alpha v) - f(x)}{\alpha}.
$$
We say that $f$ is d.d. at $x$ \textit{uniformly along finite dimensional spaces} iff $f$ is d.d. at this point and for
any $v \in X$ and finite dimensional subspace $X_0 \subset X$ one has
$$
  f'(x, v) = \lim_{[\alpha, v'] \to [+0, v], v' \in v + X_0} \frac{f(x + \alpha v') - f(x)}{\alpha},
$$
i.e. for any $\varepsilon > 0$ there exists $\delta > 0$ such that for all $\alpha > 0$ and $v' \in v + X_0$ with
$\alpha < \delta$ and $\| v' - v \| < \delta$ one has
$$
  \left| \frac{f(x + \alpha v') - f(x)}{\alpha} - f'(x, v) \right| < \varepsilon.
$$
As is easily seen, if $f$ is d.d. at $x$ and Lipschitz continuous near this point, then $f$ is d.d. at this point
uniformly along finite dimensional spaces. Furthermore, note that in the finite dimensional case $f$ is d.d. at $x$
uniformly along finite dimensional spaces iff 
$$
  f'(x, v) = \lim_{[\alpha, v'] \to [+0, v]} \frac{f(x + \alpha v') - f(x)}{\alpha}
  \quad \forall v \in X,
$$
i.e. iff $f$ is Hadamard d.d. at $x$ \cite{DemRub_book}. 

\begin{definition}
A function $f \colon U \to \mathbb{R}$ is called \textit{quasidifferentiable} at a point $x \in U$ iff $f$ is
d.d. at $x$ and there exists a pair $\mathscr{D} f(x) = [\underline{\partial} f(x), \overline{\partial} f(x)]$ of
convex weak$^*$ compact sets $\underline{\partial} f(x), \overline{\partial} f(x) \subset X^*$ such that
\begin{equation} \label{eq:QuasidiffDef}
  f'(x, v) = \max_{x^* \in \underline{\partial} f(x)} \langle x^*, v \rangle +
  \min_{y^* \in \overline{\partial} f(x)} \langle y^*, v \rangle \quad \forall v \in X
\end{equation}
(i.e. $f'(x, \cdot)$ can be represented as the difference of two continuous sublinear functions). The pair 
$\mathscr{D} f(x)$ is called a \textit{quasidifferential} of $f$ at $x$, while the sets $\underline{\partial} f(x)$ and
$\overline{\partial} f(x)$ are called \textit{subdifferential} and \textit{superdifferential} of $f$ at $x$
respectively. Finally, $f$ is called quasidifferentiable at $x$ uniformly along finite
dimensional spaces, if $f$ is quasidifferentiable and d.d. uniformly along finite dimensional spaces at this point. 
\end{definition}

The calculus of quasidifferentiable functions can be found in
\cite{DemyanovDixon,DemRub_book,Quasidifferentiability_book}. Here we only mention that the set of all functions that
are quasidifferentiable at a given point uniformly along finite dimensional spaces is closed under addition,
multiplication, pointwise maximum/minimum of finite families of functions, and composition with continuously
differentiable functions. Furthermore, any finite DC (difference-of-convex) function is quasidifferentiable uniformly
along finite dimensional spaces. 

Let us also note that if $f$ is a convex function and $\partial f$ is its subdifferential in the sense of convex
analysis, then the pair $[\partial f(x), 0]$ is a quasidifferential of $f$ at $x$. If $f$ is a DC function, 
i.e. $f = g - h$ for some convex functions $g$ and $h$, then the pair $[\partial g(x), - \partial h(x)]$ is a
quasidifferential of $f$ at $x$. Finally, if $f$ is locally Lipschitz continuous and regular in the sense of Clarke
(see~\cite[Definition~2.3.4]{Clarke}) at a point $x$, and $\partial_{Cl} f(x)$ is the Clarke subdifferential of $f$ at
$x$, then the pair $[\partial_{Cl} f(x), \{ 0 \}]$ is a quasidifferential of $f$ at $x$ (see
\cite{DemyanovDixon,DemRub_book,Quasidifferentiability_book} for more details).

Observe that a quasidifferential of a function $f$ is not unique. In particular, for any quasidifferential 
$\mathscr{D} f(x)$ of $f$ at $x$ and any weak$^*$ compact convex set $C \subset X^*$ the pair 
$[ \underline{\partial} f(x) + C, \overline{\partial} f(x) - C ]$ is a quasidifferential of $f$ at $x$ as well.
Therefore, there is an interesting problem to find a minimal, in some sense, quasidifferential of a given function. Some
results on this subject can be found in
\cite{Handschug,Scholtes92,PallaschkeUrbanski93,Gao98,GrzhybowskiUrbanski2006,Grzybowski2008,Grzhybowski2010,Pallaschke}
.

\begin{remark}
Throughout the article, when we say that a function $f$ is quasidifferentiable at a point $x$, we suppose that some
quasidifferential $\mathscr{D} f(x)$ of $f$ at $x$ is given and formulate all assumptions with respect to the given
quasidifferential $\mathscr{D} f(x)$. Alternatively, one can define a quasidifferential as an equivalence class, i.e.
as an infinite collection of all those pairs $[\underline{\partial} f(x), \overline{\partial} f(x)]$ for which
\eqref{eq:QuasidiffDef} holds true, and use equivalence classes (cf.~\cite{Uderzo2,Dolgopolik_AbstrConvApprox}). In the
author's opinion, this approach is rather cumbersome and we do not adopt it in this paper.
\end{remark}

\section{The Contingent Cone to a Set Defined by Quasidifferentiable Constraints}
\label{Section_ContingentCones}

In this section, we study the contingent cone to a set defined by quasidifferentiable equality and inequality
constraints and describe convex subcones of this cone in terms quasidifferentials of the constraints. The main results
of this section were largely inspired by the papers of Di et al. \cite{Di1,Di2}.

For any set $C \subset X$ and $x \in X$ denote $d(x, C) = \inf_{y \in C} \| x - y \|$. Recall that 
\textit{the contingent cone} $T_M(x)$ to a set $M \subset X$ at a point $x \in M$ consists of all those 
$v \in X$ for which $\liminf_{\alpha \to +0} d(x + \alpha v, M) / \alpha = 0$. Equivalently, $v \in T_M(x)$ iff there
exist a sequence $\{ \alpha_n \} \subset (0, + \infty)$ and a sequence $\{ v_n \} \subset X$ such that 
$\alpha_n \to +0$ and $v_n \to v$ as $n \to \infty$, and $x + \alpha_n v_n \in M$ for all $n \in \mathbb{N}$.
Note that the contingent cone need not be convex.

Our aim is to describe the cone $T_M(x)$ and/or its convex subcones in the case when
\begin{equation} \label{QuasidiffSet}
  M = \Big\{ x \in X \Bigm| f_i(x) = 0, \quad i \in I, \quad g_j(x) \le 0, \quad j \in J \Big\}
\end{equation}
in terms of quasidifferentials of the functions $f_i \colon X \to \mathbb{R}$ and $g_j \colon X \to \mathbb{R}$ 
(here $I = \{ 1, \ldots, m \}$ and $J = \{ 1, \ldots, l \}$). To this end, we utilize the following auxiliary result,
which is a simple corollary to the Borsuk-Krasnoselskii antipodal theorem (see, e.g. \cite[Corollary~16.7]{Zeidler}).

\begin{lemma}[generalized intermediate value theorem]
Let $r^i \colon [-1, 1]^m \to \mathbb{R}$, $i \in I = \{ 1, \ldots, m \}$, be continuous functions such that for any 
$i \in I$ and for all $\tau^j \in [-1, 1]$, $j \ne i$ one has
\begin{equation} \label{eq:AntipodalConditions}
\begin{split}
  r^i(\tau^1, \ldots, \tau^{i - 1}, -1, \tau^{i + 1}, \ldots, \tau^m) &< 0, \\
  r^i(\tau^1, \ldots, \tau^{i - 1}, 1, \tau^{i + 1}, \ldots, \tau^m) &> 0.
\end{split}
\end{equation}
Then there exists $\widehat{\tau} \in (-1, 1)^m$ such that $r^i(\widehat{\tau}) = 0$ for all $i \in I$.
\end{lemma}

For any $C \subset X^*$ and $v \in X$ denote by $s(C, v) = \sup_{x^* \in C} \langle x^*, v \rangle$ \textit{the support
function} of the set $C$. Define also $J(x) = \{ j \in J \mid g_j(x) = 0 \}$ for any $x \in X$. The following theorem
describes how one can compute a convex subcone of $T_M(x)$, if a certain constraint qualification is satisfied for
\textit{some} elements of quasidifferentials of the functions $f_i$ and $g_j$.

\begin{theorem} \label{thrm:ContingConeToQuasidiffSet}
Let the functions $f_i$, $i \in I$, be continuous in a neighbourhood of a point $\overline{x} \in M$, the functions
$g_j$, $j \notin J(\overline{x})$, be upper semicontinuous (u.s.c.) at this point, and let $f_i$, $i \in I$, and $g_j$, 
$j \in J(\overline{x})$, be quasidifferentiable at $\overline{x}$ uniformly along finite dimensional spaces. Let also 
$x_i^* \in \underline{\partial} f_i(\overline{x})$, $y_i^* \in \overline{\partial} f_i(\overline{x})$, $i \in I$,
and $z_j^* \in \overline{\partial} g_j(\overline{x})$, $j \in J(\overline{x})$, be given. Suppose finally that the
following constraint qualification holds true:
\begin{enumerate}
\item{for any $i \in I$ there exists $v_i \in X$ such that 
$s( \underline{\partial} f_i(\overline{x}) + y_i^*, v_i) < 0$ and for any $k \ne i$ one has
$s( \underline{\partial} f_k(\overline{x}) + y_k^*, v_i) \le 0$ and
$s( - x_k^* - \overline{\partial} f_k(\overline{x}), v_i) \le 0$;
\label{Assumpt_ImplicitMFCQ1}}

\item{for any $i \in I$ there exists $w_i \in X$ such that 
$s(- x_i^* - \overline{\partial} f_i(\overline{x}), w_i) < 0$ and for any $k \ne i$ one has
$s(- x_k^* - \overline{\partial} f_k(\overline{x}), w_i) \le 0$ and
$s(\underline{\partial} f_k(\overline{x}) + y_k^*, w_i) \le 0$;
\label{Assumpt_ImplicitMFCQ2}}

\item{there exists $v_0 \in X$ such that $s( \underline{\partial} g_j(\overline{x}) + z_j^*, v_0) < 0$ for any 
$j \in J(\overline{x})$, while for any $i \in I$ one has 
$s( \underline{\partial} f_i(\overline{x}) + y_i^*, v_0) \le 0$ and 
$s( - x_i^* - \overline{\partial} f_i(\overline{x}), v_0) \le 0$.
\label{Assumpt_ImplicitMFCQ3}}
\end{enumerate}
Then
\begin{multline} \label{ConvexSubcone}
  \Big\{ v \in X \Bigm| s\big( \underline{\partial} f_i(\overline{x}) + y_i^*, v \big) \le 0, \quad
  s\big( - x_i^* - \overline{\partial} f_i(\overline{x}), v \big) \le 0 \quad \forall i \in I, \\
  s\big( \underline{\partial} g_j(\overline{x}) + z_j^*, v \big) \le 0 \quad \forall j \in J(\overline{x}) \Big\}
  \subseteq T_M(\overline{x}).
\end{multline}
\end{theorem}

\begin{proof}
For all $\tau = (\tau^1, \ldots, \tau^m) \in [-1, 1]^m$ define
$$
  \eta(\tau) = \sum_{i = 1}^m ( \max\{ -\tau^i, 0 \} v_i + \max\{ \tau^i, 0 \} w_i ).
$$
For any $i \in I$ denote $p_i(\cdot) = s( \underline{\partial} f_i(\overline{x}) + y_i^*, \cdot )$ and
$q_k(\cdot) = s( - x_k^* - \overline{\partial} f_k(\overline{x}), \cdot)$. Observe that from the definition of
quasidifferential it follows that for all $v \in X$ one has $- q_i(v) \le f_i'(\overline{x}, v) \le p_i(v)$
(see~\eqref{eq:QuasidiffDef}).

Let $v \in X$ belong to the set on the left-hand side of \eqref{ConvexSubcone}. Taking into account
assumptions~\ref{Assumpt_ImplicitMFCQ1}--\ref{Assumpt_ImplicitMFCQ3} and the fact that the functions $p_i$ are sublinear
one obtains that for any $i \in I$, $n \in \mathbb{N}$, $\gamma > 0$, and $\tau \in [-1, 1]^m$ the following
inequalities hold true:
\begin{multline} \label{DD_UpperEstim}
  f'_i\Big(\overline{x}, v + \gamma v_0 
  + \frac{1}{n} \eta(\tau^1, \ldots, \tau^{i-1}, -1, \tau^{i + 1}, \ldots, \tau^m) \Big) \\
  \le p_i\Big( v + \gamma v_0 + \frac{1}{n} \eta(\tau^1, \ldots, \tau^{i-1}, -1, \tau^{i + 1}, \ldots, \tau^m) \Big) 
  \le p_i(v) + \gamma p_i(v_0) \\
  + \frac{1}{n} p_i(v_i) + 
  \frac{1}{n} \sum_{j \ne i} \big( \max\{ -\tau^j, 0 \} p_i(v_j) + \max\{ \tau^j, 0 \} p_i(w_j) \big) \\
  \le \frac{1}{n} p_i(v_i) < 0.
\end{multline}
Similarly, for any $i \in I$, $n \in \mathbb{N}$, $\gamma > 0$, and $\tau \in [-1, 1]^m$ one has
\begin{multline} \label{DD_LowerEstim}
  f'_i\Big( \overline{x}, v + \gamma v_0 
  + \frac{1}{n} \eta(\tau^1, \ldots, \tau^{i-1}, 1, \tau^{i + 1}, \ldots, \tau^m) \Big) \\
  \ge - q_i\Big( \overline{x}, v + \gamma v_0 
  + \frac{1}{n} \eta(\tau^1, \ldots, \tau^{i-1}, 1, \tau^{i + 1}, \ldots, \tau^m) \Big) \ge - \frac{1}{n} q_i(w_i) > 0.
\end{multline}
Let us verify that from \eqref{DD_UpperEstim} and \eqref{DD_LowerEstim} it follows that for any $n \in \mathbb{N}$ and
$\gamma > 0$ there exists $\alpha_n(\gamma) > 0$ such that for all $0 < \alpha < \alpha_n(\gamma)$, $i \in I$, and
$\tau \in [-1, 1]^m$ the following inequalities hold true:
\begin{align} \label{AntipodalCond_Bottom}
  f_i\left( \overline{x} + \alpha \Big( v + \gamma v_0 
  + \frac{1}{n}\eta(\tau^1, \ldots, \tau^{i-1}, -1, \tau^{i + 1}, \ldots, \tau^m) \Big) \right) &< 0, \\
  f_i\left( \overline{x} + \alpha \Big( v + \gamma v_0 + 
  \frac{1}{n}\eta(\tau^1, \ldots, \tau^{i-1}, 1, \tau^{i + 1}, \ldots, \tau^m) \Big) \right) &> 0.
  \label{AntipodalCond_Top}
\end{align}
Indeed, fix any $i \in I$, $\gamma > 0$, and $n \in \mathbb{N}$. Arguing by reductio ad absurdum, suppose that for any
$\alpha_n(\gamma) > 0$ there exist $\alpha \in (0, \alpha_n(\gamma))$ and $\tau \in [-1, 1]^m$ such that, say,
\eqref{AntipodalCond_Bottom} is not valid. Then there exist a sequence $\{ \alpha_k \} \subset (0, + \infty)$
converging to zero and a sequence $\{ \tau_k \} \subset [-1, 1]^m$ such that
$$
  f_i\left( \overline{x} + \alpha_k \Big( v + \gamma v_0 + 
  \frac{1}{n}\eta(\tau^1_k, \ldots, \tau^{i-1}_k, -1, \tau^{i + 1}_k, \ldots, \tau^m_k) \Big) \right) \ge 0
$$
Without loss of generality one can suppose that $\{ \tau_k \}$ converges to some $\widehat{\tau} \in [-1, 1]^m$.
Therefore, utilizing the facts that $f_i$ is d.d. at $\overline{x}$ uniformly along finite dimensional spaces,
the function $\eta(\cdot)$ is continuous and takes values in the finite dimensional space 
$X_0 = \linhull\{ v_i, w_i \mid i \in I \}$, and $f_i(\overline{x}) = 0$ one obtains that
\begin{align*}
  &f_i'\left( \overline{x}, v + \gamma v_0 + \frac{1}{n}
  \eta(\widehat{\tau}^1, \ldots, \widehat{\tau}^{i-1}, -1, \widehat{\tau}^{i + 1}, \ldots, \widehat{\tau}^m) \right) 
  = \\
  &\lim_{k \to \infty} \frac{1}{\alpha_k}
   f_i\left( \overline{x} + \alpha_k \Big( v + \gamma v_0 + 
  \frac{1}{n}\eta(\tau^1_k, \ldots, \tau^{i-1}_k, -1, \tau^{i + 1}_k, \ldots, \tau^m_k) \Big) \right) \ge 0,
\end{align*}
which contradicts \eqref{DD_UpperEstim}.

For any $j \in J(\overline{x})$ denote $u_j(\cdot) = s(\underline{\partial} g_j(\overline{x}) + z_j^*, \cdot)$. Note
that $u_j$ are continuous sublinear functions (recall that $\underline{\partial} g_j(\overline{x})$ is a convex
weak${}^*$ compact set). Therefore, for any $j \in J(\overline{x})$, $\gamma > 0$, $n \in \mathbb{N}$, and 
$\tau \in [-1, 1]^m$ one has 
\begin{multline*}
  g_j'\Big( \overline{x}, v + \gamma v_0 + \frac{1}{n}\eta(\tau) \Big) 
  \le u_j\Big( v + \gamma v_0 + \frac{1}{n}\eta(\tau) \Big) \\
  \le u_j(v) + \gamma u_j(v_0) + \frac{1}{n} u_j(\eta(\tau))
  \le \gamma u_j(v_0) + \frac{1}{n} \max_{s \in [-1, 1]^m} u_j(\eta(s))
\end{multline*}
(here we used the fact that $u_j(v) \le 0$, since $v$ belongs to the set on the left-hand side of
\eqref{ConvexSubcone}). By assumption~\ref{Assumpt_ImplicitMFCQ3} one has $u_j(v_0) < 0$. Consequently, for any 
$\gamma > 0$ one can find $n_{\gamma} \in \mathbb{N}$ such that for all $j \in J(\overline{x})$ and $n \ge n_{\gamma}$
one has
\begin{equation} \label{IneqConstrDD_UpperEstim}
  g_j'\left( \overline{x}, v + \gamma v_0 + \frac{1}{n}\eta(\tau) \right) \le \frac{\gamma}{2} u_j(v_0) < 0
  \quad \forall \tau \in [-1, 1]^m.
\end{equation}
Let us check that this inequality implies that for any $\gamma > 0$ and $n \ge n_{\gamma}$ there exists
$\beta_n(\gamma) > 0$ such that for all $j \in J(\overline{x})$, $\tau \in [-1, 1]^m$, and 
$0 < \alpha < \beta_n(\gamma)$ one has
\begin{equation} \label{InteriorDirections}
  g_j\left( \overline{x} + \alpha \Big( v + \gamma v_0 + \frac{1}{n}\eta(\tau) \Big) \right) < 0.
\end{equation}
Indeed, fix any $j \in J(\overline{x})$, $\gamma > 0$, and $n \ge n_{\gamma}$. Arguing by reductio ad absurdum, suppose
that for any $\beta_n(\gamma) > 0$ there exist $\alpha \in (0, \beta_n(\gamma))$ and $\tau \in [-1, 1]^m$ such that 
\eqref{InteriorDirections} is not valid. Then there exist a sequence $\{ \alpha_k \} \subset (0, + \infty)$
converging to zero and a sequence $\{ \tau_k \} \subset [-1, 1]^m$ such that
$$
  g_j\left( \overline{x} + \alpha_k \Big( v + \gamma v_0 + \frac{1}{n}\eta(\tau_k) \Big) \right) \ge 0
  \quad \forall k \in \mathbb{N}.
$$
Without loss of generality one can suppose that $\{ \tau_k \}$ converges to some $\widehat{\tau} \in [-1, 1]^m$. Hence
with the use of the facts that $g_j$ is d.d. at $\overline{x}$ uniformly along finite dimensional spaces,
the function $\eta(\cdot)$ is continuous and takes values in the finite dimensional space 
$X_0 = \linhull\{ v_i, w_i \mid i \in I \}$, and $g_i(\overline{x}) = 0$, since $j \in J(\overline{x})$, one obtains
that
$$
  g_j'\left( \overline{x}, v + \gamma v_0 + \frac{1}{n}\eta(\widehat{\tau}) \right)
  = \lim_{k \to \infty} \frac{1}{\alpha_k}
  g_j\left( \overline{x} + \alpha_k \Big( v + \gamma v_0 + \frac{1}{n}\eta(\tau_k) \Big) \right) \ge 0,
$$
which contradicts \eqref{IneqConstrDD_UpperEstim}.

By our assumptions the functions $f_i$ are continuous in a neighbourhood $U$ of $\overline{x}$. By virtue of the fact
that the set $\{ \eta(\tau) \in X \mid \tau \in [-1, 1]^m \}$ is compact, for any $n \in \mathbb{N}$ and $\gamma > 0$
one can find $\delta_n(\gamma) > 0$ such that 
\begin{equation} \label{WithinContinuityRegion}
  \left\{ \overline{x} + \alpha \Big( v + \gamma v_0 + \frac{1}{n}\eta(\tau) \Big) \in X \Bigm| 
  \alpha \in [0, \delta_n(\gamma)], \: \tau \in [-1, 1]^m  \right\} \subset U.
\end{equation}
Furthermore, choosing $\delta_n(\gamma)$ small enough one can suppose that $g_j(x) < 0$ for any 
$j \notin J(\overline{x})$ and $x$ from the set on the left-hand side of \eqref{WithinContinuityRegion}, since 
$g_j(\overline{x}) < 0$ for any such $j$ and these functions are u.s.c. at $\overline{x}$.

Fix $\gamma > 0$, and for any $n \ge n_{\gamma}$ choose 
$0 < \alpha_n < \min\{ \alpha_n(\gamma), \beta_n(\gamma), \delta_n(\gamma) \}$ such that $\alpha_n \to 0$ as 
$n \to \infty$. For any $i \in I$ and $n \ge n_{\gamma}$ define
$$
  r^i_n(\tau) = f_i\left( \overline{x} + \alpha_n \Big( v + \gamma v_0 + \frac{1}{n}\eta(\tau) \Big) \right)
  \quad \forall \tau \in [-1, 1]^m.
$$
From \eqref{WithinContinuityRegion} and the definition of $U$ it follows that the functions $r^i_n(\cdot)$, 
$i \in I$, are continuous. Furthermore, inequalities \eqref{AntipodalCond_Bottom} and \eqref{AntipodalCond_Top} imply
that the functions $r^i_n(\cdot)$, $i \in I$, satisfy inequalities \eqref{eq:AntipodalConditions} from the generalized
intermediate value theorem. Therefore, by this theorem for any $n \ge n_{\gamma}$ there exists 
$\widehat{\tau}_n \in (-1, 1)^m$ such that $r^i_n(\widehat{\tau}_n) = 0$ for all $i \in I$, i.e. 
$f_i(\overline{x} + \alpha_n v_n) = 0$ for any $i \in I$, where 
$v_n = v + \gamma v_0 + \eta(\widehat{\tau}_n) / n$. Moreover, by \eqref{InteriorDirections} and the
choice of $\delta_n(\gamma)$ one has $g_j(\overline{x} + \alpha_n v_n) < 0$ for all $j \in J$. Thus, 
$\overline{x} + \alpha_n v_n \in M$ for any $n \ge n_{\gamma}$. Hence with the use of the fact
that $v_n \to v + \gamma v_0$ as $n \to \infty$ one obtains that $v + \gamma v_0 \in T_M(\overline{x})$ for any 
$\gamma > 0$, which implies that $v \in T_M(\overline{x})$, since the contingent cone is always closed. Thus, the proof
is complete.
\end{proof}

Observe that the set on the left-hand side of \eqref{ConvexSubcone} is a nontrivial closed convex cone ($v_0$ belongs to
this cone). Thus, the theorem above provides one with a way to compute convex subcones of the contingent cone
$T_M(\overline{x})$ with the use of those vectors from quasidifferentials of the functions $f_i$ and $g_j$ that satisfy
assumptions \ref{Assumpt_ImplicitMFCQ1}--\ref{Assumpt_ImplicitMFCQ3}. Let us give a simple geometric description of
these assumptions, which sheds some light on the way they are connected with well-known constraint qualifications.

\begin{remark}
It is worth noting that there is a connection between assumptions
\ref{Assumpt_ImplicitMFCQ1}--\ref{Assumpt_ImplicitMFCQ3} of Theorem~\ref{thrm:ContingConeToQuasidiffSet} and some
conditions on the directional derivatives of the functions $f_i$ and $g_j$. Indeed, from the definition of
quasidifferential \eqref{eq:QuasidiffDef} it follows that assumption~\ref{Assumpt_ImplicitMFCQ1} is satisfied
\textit{for
some} $x_i^* \in \underline{\partial} f_i(\overline{x})$ and $y_i^* \in \overline{\partial} f_i(\overline{x})$, 
$i \in I$, if and only if 
\begin{equation} \label{ImplicitMFCQviaDD1}
  f'_i(\overline{x}, v_i) < 0, \quad f'_k(\overline{x}, v_i) = 0 \quad \forall k \ne i.
\end{equation}
Similarly, assumption~\ref{Assumpt_ImplicitMFCQ2} is satisfied  \textit{for some} 
$x_i^* \in \underline{\partial} f_i(\overline{x})$ and $y_i^* \in \overline{\partial} f_i(\overline{x})$, $i \in I$
(which might differ from the ones for which assumption~\ref{Assumpt_ImplicitMFCQ1} is valid) if and only if 
\begin{equation} \label{ImplicitMFCQviaDD2}
  f'_i(\overline{x}, w_i) > 0, \quad f'_k(\overline{x}, w_i) = 0 \quad \forall k \ne i.
\end{equation}
Finally, assumption~\ref{Assumpt_ImplicitMFCQ3} holds true \textit{for some} $x_i^* \in \underline{\partial}
f_i(\overline{x})$, $y_i^* \in \overline{\partial} f_i(\overline{x})$, $i \in I$, and 
$z_j^* \in \overline{\partial} g_j(\overline{x})$, $j \in J(\overline{x})$, if and only if
\begin{equation} \label{ImplicitMFCQviaDD3}
  g'_j(\overline{x}, v_0) < 0 \quad \forall j \in J(\overline{x}), \quad
  f'_i(\overline{x}, v_0) = 0 \quad \forall i \in I.
\end{equation}
Note, howerever, that the validity of \eqref{ImplicitMFCQviaDD1}--\eqref{ImplicitMFCQviaDD3} does not imply that 
\ref{Assumpt_ImplicitMFCQ1}--\ref{Assumpt_ImplicitMFCQ3} holds true, since
\eqref{ImplicitMFCQviaDD1}--\eqref{ImplicitMFCQviaDD3} only imply that each of assumptions
\ref{Assumpt_ImplicitMFCQ1}--\ref{Assumpt_ImplicitMFCQ3} is valid for some $x_i^*$, $y_i^*$, and $z_j^*$, while in
Theorem~\ref{thrm:ContingConeToQuasidiffSet} we must suppose that they are valid for the same $x_i^*$, $y_i^*$, and
$z_j^*$.
\end{remark}

For any subset $A$ of a real vector space $E$ denote by
$$
  \cone A = \left\{ \sum_{i = 1}^n \lambda_i x_i \Biggm| 
  x_i \in A, \enspace \lambda_i \ge 0, \enspace i \in \{ 1, \ldots, n \}, \enspace n \in \mathbb{N} \right\}
$$
the smallest convex cone containing $A$ and by $\linhull(A)$ be the linear span of $A$.

\begin{proposition} \label{prp:GeometricRegConditions}
Let the functions $f_i$, $i \in I$, and $g_j$, $j \in J(\overline{x})$, be quasidifferentiable at a point 
$\overline{x} \in M$. Let also $x_i^* \in \underline{\partial} f_i(\overline{x})$, 
$y_i^* \in \overline{\partial} f_i(\overline{x})$, $i \in I$, and $z_j^* \in \overline{\partial} g_j(\overline{x})$, 
$j \in J(\overline{x})$, be given. Then assumptions~\ref{Assumpt_ImplicitMFCQ1}--\ref{Assumpt_ImplicitMFCQ3} of
Theorem~\ref{thrm:ContingConeToQuasidiffSet} are satisfied if and only if
\begin{gather}	\label{AlmostLI_SubQuasidiff}
  C_i \cap \cl^*\cone \big\{ - C_k \bigm| k \ne i \big\} = \emptyset \quad \forall i \in I, \\
  \co\big\{ \underline{\partial} g_j(\overline{x}) + z_j^* \bigm| j \in J(\overline{x}) \big\} \cap
  \cl^* \cone \big\{ - C_i \bigm| i \in I \big\} = \emptyset,	\label{NonIntersect_SubQuasidiff}
\end{gather}
where 
$C_i = ( \underline{\partial} f_i(\overline{x}) + y_i^*) \cup (- x_i^* - \overline{\partial} f_i(\overline{x}) )$, 
$i \in I$.
\end{proposition}

\begin{proof}
Let assumption~\ref{Assumpt_ImplicitMFCQ3} from Theorem~\ref{thrm:ContingConeToQuasidiffSet} be valid. Then, as is easy
to see, $\langle x^*, v_0 \rangle < 0$ for any 
$x^* \in \co\{ \underline{\partial} g_j(\overline{x}) + z_j^* \mid j \in J(\overline{x}) \}$, while
$\langle x^*, v_0 \rangle \ge 0$ for any $x^* \in \cl^* \cone\{ - C_i \mid i \in I \}$. Hence
\eqref{NonIntersect_SubQuasidiff} holds true. Conversely, if \eqref{NonIntersect_SubQuasidiff} holds true, then
applying the separation theorem in the space $X^*$ endowed with weak${}^*$ topology one can find $v_0$ satisfying 
assumption~\ref{Assumpt_ImplicitMFCQ3}. Thus, this assumption is equivalent to \eqref{NonIntersect_SubQuasidiff}.
 
Let now assumption~\ref{Assumpt_ImplicitMFCQ1} of Theorem~\ref{thrm:ContingConeToQuasidiffSet} be satisfied. Then
$\langle x^*, v_i \rangle < 0$ for any $x_* \in \underline{\partial} f_i(\overline{x}) + y_i^*$, while
$\langle x^*, v_i \rangle \ge 0$ for any $x^* \in \cl^*\cone\{ - C_k \mid k \ne i \}$, which implies that the sets
$\underline{\partial} f_i(\overline{x}) + y_i^*$ and $\cl^*\cone\{ - C_k \mid k \ne i \}$ do not intersect.
Conversely, if these sets do not intersect, then applying the separation theorem in the space $X^*$ endowed with
weak${}^*$ topology one can find $v_i$ satisfying assumption~\ref{Assumpt_ImplicitMFCQ1}.

Arguing in the same way one can check that assumption~\ref{Assumpt_ImplicitMFCQ2} of
Theorem~\ref{thrm:ContingConeToQuasidiffSet} is satisfied iff
the sets $- x_i^* - \overline{\partial} f_i(\overline{x})$ and $\cl^*\cone\{ - C_k \mid k \ne i \}$ do not
have common points. Thus, assumptions~\ref{Assumpt_ImplicitMFCQ1} and \ref{Assumpt_ImplicitMFCQ2} are equivalent to
\eqref{AlmostLI_SubQuasidiff}.
\end{proof}

\begin{remark}
One can readily check that in the case $I = \{ 1 \}$ (i.e. when there is only one equality constraint) condition
\eqref{AlmostLI_SubQuasidiff} is reduced to the assumption that
$0 \notin C_1$, i.e. $0 \notin \underline{\partial} f_i(\overline{x}) + y_i^*$ and
$0 \notin - x_i^* - \overline{\partial} f_i(\overline{x})$.
\end{remark}

Let a function $f \colon X \to \mathbb{R}$ be quasidifferentiable at a point $x \in X$. Denote by
$[\mathscr{D} f(x)]^+ = \underline{\partial} f(x) + \overline{\partial} f(x)$ \textit{a quasidifferential sum} of $f$
at $x$. Quasidifferential sum is a weak${}^*$ compact convex set, which, as is easy to see, is not invariant under the
choice of quasidifferential. See \cite{Uderzo2,Dolgopolik_MetricReg} for applications of quasidifferential sum to
nonsmooth optimization and related problems.

Recall that subsets $A_1, \ldots, A_m$ of a real vector space $E$ are said to be \textit{linearly independent}, if the
inclusion $0 \in \lambda_1 A_1 + \ldots \lambda_m A_m$ is valid iff $\lambda_1 = \ldots = \lambda_m = 0$. We say that
these sets are \textit{strongly linearly independent}, if $A_i \cap \linhull\{ A_k \mid k \ne i \} = \emptyset$ for all
$i \in \{ 1, \ldots, m \}$. Clearly, if the sets $A_i$ are strongly linearly independent, they are linearly independent;
however the converse implication does not hold true in the general case (take $E = \mathbb{R}^2$, 
$A_1 = \co\{ (\pm 1, 1)^T \}$, and $A_2 = \{ (1, 0)^T \}$). In the case $m = 1$, (strong) linear independence is reduced
to the assumption that $0 \notin A_1$. 

\begin{proposition} \label{prp:Equiv_qdMFCQ}
Let $f_i$ and $g_j$ be as in Proposition~\ref{prp:GeometricRegConditions}. Then for assumptions
\ref{Assumpt_ImplicitMFCQ1}--\ref{Assumpt_ImplicitMFCQ3} of Theorem~\ref{thrm:ContingConeToQuasidiffSet} to be satisfied
for all $x_i^* \in \underline{\partial} f_i(\overline{x})$, $y_i^* \in \overline{\partial} f_i(\overline{x})$, 
$i \in I$, and $z_j^* \in \overline{\partial} g_j(\overline{x})$, $j \in J(\overline{x})$ it is sufficient that
\begin{gather}	\label{AlmostLI_qdMFCQ}
  [\mathscr{D} f_i(\overline{x})]^+ \cap 
  \cl^* \linhull \big\{ [\mathscr{D} f_k(\overline{x})]^+ \bigm| k \ne i \big\} = \emptyset
  \quad \forall i \in I, \\
  \co\big\{ [\mathscr{D} g_j(\overline{x})]^+ \bigm| j \in J(\overline{x}) \big\} \cap
  \cl^*\linhull \big\{ [\mathscr{D} f_i(\overline{x})]^+ \bigm| i \in I \big\} = \emptyset.
  \label{NonIntersect_qdMFCQ}
\end{gather}
Furthermore, these conditions become necessary, if the spans in \eqref{NonIntersect_qdMFCQ} and \eqref{AlmostLI_qdMFCQ}
are weak$^*$ closed (in particular, if $X$ is finite dimensional). In addition, if the span in \eqref{AlmostLI_qdMFCQ}
is
weak${}^*$ closed for any $i \in I$, then conditions \eqref{NonIntersect_qdMFCQ} and \eqref{AlmostLI_qdMFCQ} are
satisfied
if and only if the Mangasarian-Fromovitz constraint qualification in terms of quasidifferentials (q.d.-MFCQ) holds true
at $\overline{x}$, i.e. the sets $[\mathscr{D} f_i(\overline{x})]^+$, $i \in I$, are strongly linearly independent and
there exists $v_0 \in X$ such that $\langle x^*, v_0 \rangle = 0$ for any $x^* \in [\mathscr{D} f_i(\overline{x})]^+$, 
$i \in I$, and $\langle x^*, v_0 \rangle < 0$ for any $x^* \in [\mathscr{D} g_j(\overline{x})]^+$, 
$j \in J(\overline{x})$.
\end{proposition}

\begin{proof}
Let conditions \eqref{NonIntersect_qdMFCQ} and \eqref{AlmostLI_qdMFCQ} be satisfied. Fix any 
$x_i^* \in \underline{\partial} f_i(\overline{x})$, $y_i^* \in \overline{\partial} f_i(\overline{x})$, 
$i \in I$, and $z_j^* \in \overline{\partial} g_j(\overline{x})$, $j \in J(\overline{x})$, and denote
$C_i = ( \underline{\partial} f_i(\overline{x}) + y_i^*) \cup (- x_i^* - \overline{\partial} f_i(\overline{x}) )$.
From the definition of quasidifferential sum it follows that
\begin{gather*}
  \cone \big\{ - C_i \bigm| i \in I_0 \big\} \subseteq 
  \linhull \big\{ [\mathscr{D} f_i(\overline{x})]^+ \bigm| i \in I_0 \big\}, \\
  \co\big\{ \underline{\partial} g_j(\overline{x}) + z_j^* \bigm| j \in J(\overline{x}) \big\}
  \subseteq \co\big\{ [\mathscr{D} g_j(\overline{x})]^+ \bigm| j \in J(\overline{x}) \big\}
\end{gather*}
for any $I_0 \subseteq I$. Therefore, \eqref{NonIntersect_qdMFCQ} implies \eqref{NonIntersect_SubQuasidiff}. Observe
also
that \eqref{AlmostLI_qdMFCQ} is satisfied iff
$(- [\mathscr{D} f_i(\overline{x})]^+) \cap 
\cl^* \linhull \big\{ [\mathscr{D} f_k(\overline{x})]^+ \bigm| k \ne i \big\} = \emptyset$ for all $i \in I$.
Furthermore, $C_i \subseteq [\mathscr{D} f_i(\overline{x})]^+ \cup (-[\mathscr{D} f_i(\overline{x})]^+)$ for any 
$i \in I$ by definitions. Therefore, \eqref{AlmostLI_qdMFCQ} implies \eqref{AlmostLI_SubQuasidiff}. Hence applying
Proposition~\ref{prp:GeometricRegConditions} one obtains that assumptions
\ref{Assumpt_ImplicitMFCQ1}--\ref{Assumpt_ImplicitMFCQ3} of Theorem~\ref{thrm:ContingConeToQuasidiffSet} are satisfied
for all $x_i^* \in \underline{\partial} f_i(\overline{x})$, $y_i^* \in \overline{\partial} f_i(\overline{x})$, 
$i \in I$, and $z_j^* \in \overline{\partial} g_j(\overline{x})$, $j \in J(\overline{x})$.

Suppose now that the spans in \eqref{NonIntersect_qdMFCQ} and \eqref{AlmostLI_qdMFCQ} are weak$^*$ closed, and
assumptions \ref{Assumpt_ImplicitMFCQ1}--\ref{Assumpt_ImplicitMFCQ3} of Theorem~\ref{thrm:ContingConeToQuasidiffSet}
are satisfied for all $x_i^* \in \underline{\partial} f_i(\overline{x})$, 
$y_i^* \in \overline{\partial} f_i(\overline{x})$, $i \in I$, and $z_j^* \in \overline{\partial} g_j(\overline{x})$, 
$j \in J(\overline{x})$. Arguing by reductio ad absurdum, suppose that either \eqref{NonIntersect_qdMFCQ} or
\eqref{AlmostLI_qdMFCQ} does not hold true. Suppose at first that \eqref{NonIntersect_qdMFCQ} is not valid. Applying the
definitions of linear span and convex conic hull one can verify that
$$
  \linhull \big\{ [\mathscr{D} f_i(\overline{x})]^+ \bigm| i \in I \big\} =
  \sum_{i \in I} \cone [\mathscr{D} f_i(\overline{x})]^+ 
  + \sum_{i \in I} \cone\Big\{ - [\mathscr{D} f_i(\overline{x})]^+ \Big\}.
$$
Hence for any $j \in J(\overline{x})$ there exist $h_j^* \in \underline{\partial} g_j(\overline{x})$, 
$z_j^* \in \overline{\partial} g_j(\overline{x})$, and $\alpha_j \ge 0$, while for any $i \in I$ there exist 
$x_i^*, \widehat{x}_i^* \in \underline{\partial} f_i(\overline{x})$, 
$y_i^*, \widehat{y}_i^* \in \overline{\partial} f_i(\overline{x})$, and $\lambda_i, \mu_i \ge 0$ such that
$$
  \sum_{j \in J(\overline{x})} \alpha_j (h_j^* + z_j^*) = 
  \sum_{i \in I} \lambda_i ( x_i^* + \widehat{y}_i^* )
  - \sum_{i \in I} \mu_i ( \widehat{x}_i^* + y_i^* ),
$$
and $\sum_{j \in J(\overline{x})} \alpha_j = 1$ (here we used the fact that 
$\cone [\mathscr{D} f_i(\overline{x})]^+ = \bigcup_{t \ge 0} t [\mathscr{D} f_i(\overline{x})]^+$, since
$[\mathscr{D} f_i(\overline{x})]^+$ is a convex set). Therefore
$$
  \co\big\{ \underline{\partial} g_j(\overline{x}) + z_j^* \bigm| j \in J(\overline{x}) \big\} \cap
  \cone\big\{ x_i + \overline{\partial} f_i(\overline{x}), \: 
  - \underline{\partial} f_i(\overline{x}) - y_i^* \bigm| i \in I \big\} \ne \emptyset,
$$
which is impossible by Proposition~\ref{prp:GeometricRegConditions}. Arguing in a similar way one can check that
if \eqref{AlmostLI_qdMFCQ} is not valid, then there exists $x_i^* \in \underline{\partial} f_i(\overline{x})$, 
$y_i^* \in \overline{\partial} f_i(\overline{x})$, $i \in I$ for which \eqref{AlmostLI_SubQuasidiff} does not holds
true, which is, once again, impossible by Proposition~\ref{prp:GeometricRegConditions}.

It remains to note that if the span in \eqref{AlmostLI_qdMFCQ} is weak$^*$ closed for all $i \in I$, then by definition
\eqref{AlmostLI_qdMFCQ} means that the sets $[\mathscr{D} f_i(\overline{x})]^+$, $i \in I$, are strongly linearly
independent. In turn, \eqref{NonIntersect_qdMFCQ} implies the validity of the second condition of q.d.-MFCQ 
(the existence of $v_0$) by the separation theorem, while the validity of the converse implication follows directly from
definitions.
\end{proof}

\begin{remark} \label{rmrk:qdMFCQ}
A weak form of the Mangasarian-Fromovitz constraint qualification in terms of quasidifferentials, in which the strong
linear independence of $[\mathscr{D} f_i(\overline{x})]^+$, $i \in I$, is replaced by their linear independence, was
first introduced by the author in \cite{Dolgopolik_MetricReg} for an analysis of the metric regularity of
quasidifferentiable mappings. In the case when $X = \mathbb{R}^n$ and there are no equality constraints it was utilized
in \cite{LudererRosiger90,LudererRosiger91,KuntzScholtes} for an analysis of optimality conditions, while in the case
when $X = \mathbb{R}^n$ and there are no inequality constraints a similar condition was proposed by Demyanov
\cite{Demyanov_NewtMeth} for the study of implicit functions and a nonsmooth version of Newton's method.
\end{remark}

Let us give several simple corollaries to Theorem~\ref{thrm:ContingConeToQuasidiffSet}. At first, note that
this theorem obviously remains valid if there are no equality or there are no inequality constraints. Furthermore, 
an analysis of the proof of Theorem~\ref{thrm:ContingConeToQuasidiffSet} indicates that when there are no equality
constraints the assumption that $g_j$ are d.d. uniformly along finite dimensional spaces is unnecessary (in this case
one defines $\eta(\cdot) \equiv 0$). 

\begin{corollary} \label{crlr:ContingCone_EqualConstr}
Let the functions $f_i$, $i \in I$, be continuous in a neighbourhood of a point $\overline{x} \in M$,
quasidifferentiable at this point uniformly along finite dimensional spaces, and let $J = \emptyset$. 
Let also $x_i^* \in \underline{\partial} f_i(\overline{x})$ and $y_i^* \in \overline{\partial} f_i(\overline{x})$, 
$i \in I$ be such that \eqref{AlmostLI_SubQuasidiff} holds true (in particular, if $m = 1$, then it is sufficient to
suppose that $0 \notin \partial f_1(\overline{x}) + y_1^*$ 
and $0 \notin x_1^* + \overline{\partial} f_1(\overline{x})$). Then
$$
  \Big\{ v \in X \Bigm| s\big( \underline{\partial} f_i(\overline{x}) + y_i^*, v \big) \le 0, \:
  s\big( - x_i^* - \overline{\partial} f_i(\overline{x}), v \big) \le 0, \: i \in I \Big\}
  \subseteq T_M(\overline{x}).
$$
\end{corollary}

\begin{corollary}
Let $\overline{x} \in M$ be a given point and $I = \emptyset$. Suppose that the functions $g_j$, 
$j \in J(\overline{x})$, are quasidifferentiable at $\overline{x}$, the functions $g_j$, $j \notin J(\overline{x})$, are
upper semicontinuous (u.s.c.) at this point. Let also $z_j^* \in \overline{\partial} g_j(\overline{x})$, $j \in
J(\overline{x})$, be such that
$0 \notin \co\{ \underline{\partial} g_j(\overline{x}) + z_j^* \mid j \in J(\overline{x}) \}$. Then
$$
  \Big\{ v \in X \Bigm| s\big( \underline{\partial} g_j(\overline{x}) + z_j^*, v \big) \le 0, \:
  j \in J(\overline{x}) \Big\} \subseteq T_M(\overline{x}).
$$
\end{corollary}

Theorem~\ref{thrm:ContingConeToQuasidiffSet} can also be utilized to describe the contingent cone $T_M(\overline{x})$
in terms of the directional derivatives of the functions $f_i$ and $g_j$ in the case when these functions are Hadamard
directionally differentiable. Recall that a function $f \colon X \to \mathbb{R}$ is called \textit{Hadamard d.d.} at 
a point $x \in X$, if for any $v \in X$ there exists the finite limit
$$
  f'(x, v) = \lim_{[\alpha, v'] \to [+0, v]} \frac{f(x + \alpha v') - f(x)}{\alpha}
$$
Note that when $f$ is Hadamard d.d. at $x$, $f'(x, v)$ coincides with the usual directional derivative.

\begin{corollary} \label{crlr:ContingetCone_via_DirectDeriv}
Let the functions $f_i$, $i \in I$, be continuous in a neighbourhood of a point $\overline{x} \in M$, the functions
$g_j$, $j \notin J(\overline{x})$, be upper semicontinuous (u.s.c.) at this point, and let $f_i$, $i \in I$, and $g_j$, 
$j \in J(\overline{x})$, be quasidifferentiable at $\overline{x}$ uniformly along finite dimensional spaces. Suppose
also that assumptions \eqref{AlmostLI_qdMFCQ} and \eqref{NonIntersect_qdMFCQ} are satisfied.
Then
\begin{equation} \label{ContingentCone_via_DirectDeriv}
  \Big\{ v \in X \Bigm| f'_i(\overline{x}, v) = 0, \: i \in I, \: 
  g'_j(\overline{x}, v) \le 0, \: j \in J(\overline{x}) \Big\} \subseteq T_M(\overline{x}),
\end{equation}
Moreover, the opposite inclusion holds true, provided $f_i$, $i \in I$, and $g_j$, $j \in J(\overline{x})$, are
Hadamard d.d. at $\overline{x}$ (in particular, if they are Lipschitz continuous near this point).
\end{corollary}

\begin{proof}
Let $v \in X$ belong to the left-hand side of \eqref{ContingentCone_via_DirectDeriv}. By the definition of
quasidifferential one has
$$
  f_i'(\overline{x}, v) = \max_{x^* \in \underline{\partial} f_i(\overline{x})} \langle x^*, v \rangle
  + \min_{y^* \in \overline{\partial} f_i(\overline{x})} \langle y^*, v \rangle
$$
(the maximum and the minimum are attained due to the fact that the sets $\underline{\partial} f_i(\overline{x})$ and
$\overline{\partial} f_i(\overline{x})$ are weak${}^*$ compact). Hence for any $i \in I$ there exist
$x_i^* \in \underline{\partial} f_i(\overline{x})$ and $y_i^* \in \overline{\partial} f_i(\overline{x})$ such that
$s(\underline{\partial} f_i(\overline{x}) + y_i^*, v) = 0$ and 
$s(- x_i^* -\overline{\partial} f_i(\overline{x}), v) = 0$. Similarly, for any $j \in J(\overline{x})$ there exists
$z_j^* \in \overline{\partial} g_j(\overline{x})$ such that 
$s(\underline{\partial} g_j(\overline{x}) + z_j^*, v) \le 0$. Consequently, applying
Proposition~\ref{prp:Equiv_qdMFCQ} and Theorem~\ref{thrm:ContingConeToQuasidiffSet} one obtains that
$v \in T_M(\overline{x})$.

Let us prove the converse inclusion. Choose $v \in T_M(\overline{x})$. By definition there exist sequences 
$\{ \alpha_n \} \subset (0, + \infty)$ and $\{ v_n \} \subset X$ such that $\alpha_n \to 0$ and $v_n \to v$ as 
$n \to + \infty$, and $\overline{x} + \alpha_n v_n \in M$ for all $n \in \mathbb{N}$. Fix any $i \in I$. By our
assumption $f_i$ is Hadamard d.d. at $\overline{x}$. Therefore
$$
  f'_i(\overline{x}, v) = \lim_{n \to \infty} \frac{f_i(\overline{x} + \alpha_n v_n) - f_i(\overline{x})}{\alpha_n} = 0,
$$
(here we used the fact that $f_i(\overline{x} + \alpha_n v_n) = 0$ for all $n \in \mathbb{N}$, since 
$\overline{x} + \alpha_n v_n \in M$). Similarly, from the fact that $\overline{x} + \alpha_n v_n \in M$ for all 
$n \in \mathbb{N}$ and the function $g_j$, $j \in J(\overline{x})$ is Hadamard d.d. at $\overline{x}$ it follows that 
$g_j'(\overline{x}, v) \le 0$. Thus, $f_i'(\overline{x}, v) = 0$ for any $i \in I$ and $g_j'(\overline{x}, v) \le 0$ for
any $j \in J(\overline{x})$, i.e. $v$ belongs to the left-hand side of \eqref{ContingentCone_via_DirectDeriv}, which
completes the proof.
\end{proof}

Let us finally present two simple examples illustrating Theorem~\ref{thrm:ContingConeToQuasidiffSet}.

\begin{example}
Let $X = \mathbb{R}^2$, $\overline{x} = 0$, and 
$$
  M = \Big\{ x = (x^1, x^2)^T \in \mathbb{R}^2 \Bigm| f(x) = |x^1| - x^2 = 0, \quad g(x) = x^1 \le 0 \Big\}.
$$
The functions $f$ and $g$ are quasidifferentiable at $\overline{x}$ and one can define
\begin{gather*}
  \underline{\partial} f(\overline{x}) = \co\left\{ \begin{pmatrix} 1 \\ -1 \end{pmatrix},
  \begin{pmatrix} - 1 \\ -1 \end{pmatrix} \right\}, \quad
  \overline{\partial} f(\overline{x}) = \left\{ \begin{pmatrix} 0 \\ 0 \end{pmatrix} \right\}, \\
  \underline{\partial} g(\overline{x}) = \left\{ \begin{pmatrix} 1 \\ 0 \end{pmatrix} \right\}, \quad
  \overline{\partial} g(\overline{x}) = \left\{ \begin{pmatrix} 0 \\ 0 \end{pmatrix} \right\}.
\end{gather*}
Observe that 
$$
  [\mathscr{D} f(\overline{x})]^+ = \co\left\{ \begin{pmatrix} 1 \\ -1 \end{pmatrix}, 
  \begin{pmatrix} -1 \\ -1 \end{pmatrix} \right\}, \quad
  [\mathscr{D} g(\overline{x})]^+ = \left\{ \begin{pmatrix} 1 \\ 0 \end{pmatrix} \right\}.
$$
Thus, $\linhull [\mathscr{D} f(\overline{x})]^+ = \mathbb{R}^2$, and q.d.-MFCQ is not satisfied at $\overline{x}$.
Nevertheless, Theorem~\ref{thrm:ContingConeToQuasidiffSet} enables us to compute the entire cone $T_M(\overline{x})$.
Indeed, put $x^* = (-1, -1)^T \in \underline{\partial} f(\overline{x})$ and
$y^* = 0 \in \overline{\partial} f(\overline{x})$. Then $0 \notin x^* + \overline{\partial} f(\overline{x})$ and
$0 \notin \underline{\partial} f(\overline{x}) + y^*$. Define $z^* = 0 \in \overline{\partial} g(\overline{x})$. Then
for $v_0 = (-1, 1)^T$ one has 
$$
  s(\underline{\partial} g(\overline{x}) + z^*, v_0) = -1 < 0, \enspace
  s(\underline{\partial} f(\overline{x}) + y^*, v_0) = 0, \enspace
  s(- x^* - \underline{\partial} f(\overline{x}), v_0) = 0.
$$
Thus, all assumptions of Theorem~\ref{thrm:ContingConeToQuasidiffSet} are satisfied for the chosen vectors $x^*$, $y^*$,
and $z^*$. Consequently, by this theorem the cone
\begin{gather*}
  \Big\{ v \in \mathbb{R}^2 \Bigm| s(\underline{\partial} f(\overline{x}) + y^*, v) \le 0, \enspace
  s(- x^* - \overline{\partial} f(\overline{x}), v) \le 0, \enspace
  s(\underline{\partial} g(\overline{x}) + z^*, v) \le 0 \Big\} \\
  = \big\{ v \in \mathbb{R}^2 \bigm| |v^1| - v^2 \le 0, \enspace v^1 + v^2 \le 0, \enspace v^1 \le 0 \big\}
  = \big\{ (-t, t)^T \in \mathbb{R}^2 \bigm| t \ge 0 \big\}
\end{gather*}
is contained in $T_M(\overline{x})$. It remains to note that, in actuality, this cone coincides with $M$ and
$T_M(\overline{x})$.
\end{example}

\begin{example} \label{example:CrossEqConstr}
Let $X = \mathbb{R}^2$, $\overline{x} = 0$, and 
$$
  M = \Big\{ x = (x^1, x^2)^T \in \mathbb{R}^2 \Bigm| f(x) = |\sin x^1| - |\sin x^2| = 0 \Big\}.
$$
Applying standard rules of quasidifferential calculus (see, e.g. \cite{DemRub_book}) one can easily check that $f$ is
quasidifferentiable at $\overline{x}$ and one can define 
$\mathscr{D} f(0) = [ \underline{\partial} f(0), \overline{\partial} f(0) ]$ with
\begin{equation} \label{SimplestDCFunc_Quasidiff}
  \underline{\partial} f(0) = \co\left\{ \begin{pmatrix} 1 \\ 0 \end{pmatrix},
  \begin{pmatrix} - 1 \\ 0 \end{pmatrix} \right\}, \quad
  \overline{\partial} f(0) = \co\left\{ \begin{pmatrix} 0 \\ 1 \end{pmatrix},
  \begin{pmatrix} 0 \\ -1 \end{pmatrix} \right\}.
\end{equation}
Observe that $[\mathscr{D} f(0)]^+ = \{ x \in \mathbb{R}^2 \mid \max\{ |x^1|, |x^2| \} \le 1 \}$, i.e. q.d.-MFCQ is not
satisfied at $\overline{x}$, since $0 \in [\mathscr{D} f(0)]^+$. Nevertheless, as in the previous example,
Theorem~\ref{thrm:ContingConeToQuasidiffSet} still allows one to compute the entire contingent cone $T_M(0)$. Indeed,
denote $x_{\pm}^* = (\pm 1, 0)^T$ and $y_{\pm}^* = ( 0, \pm 1)^T$. Clearly, 
$0 \notin \underline{\partial} f(0) + y_{\pm}^*$ and $0 \notin \overline{\partial} f(0) + x_{\pm}^*$. Therefore,
applying Corollary~\ref{crlr:ContingCone_EqualConstr} one gets that $K_i \subset T_M(0)$, $1 \le i \le 4$, where
\begin{align*}
  K_1 &= \Big\{ v \in \mathbb{R}^2 \Bigm| s(\underline{\partial} f(0) + y_+^*, v) \le 0, \:
  s(- x_+^* - \overline{\partial} f(0), v) \le 0 \Big\} \\
  &= \big\{ v \in \mathbb{R}^2 \bigm| |v^1| + v^2 \le 0, \: - v^1 + |v^2| \le 0 \big\}
  = \big\{ (t, -t)^T \in \mathbb{R}^2 \bigm| t \ge 0 \big\}, \\
  K_2 &= \Big\{ v \in \mathbb{R}^2 \Bigm| s(\underline{\partial} f(0) + y_+^*, v) \le 0, \:
  s(- x_-^* - \overline{\partial} f(0), v) \le 0 \Big\} \\
  &= \big\{ v \in \mathbb{R}^2 \bigm| |v^1| + v^2 \le 0, \: v^1 + |v^2| \le 0 \big\}
  = \big\{ (-t, -t)^T \in \mathbb{R}^2 \bigm| t \ge 0 \big\}, \\
  K_3 &= \Big\{ v \in \mathbb{R}^2 \Bigm| s(\underline{\partial} f(0) + y_-^*, v) \le 0, \:
  s(- x_+^* - \overline{\partial} f(0), v) \le 0 \Big\} \\
  &= \big\{ v \in \mathbb{R}^2 \bigm| |v^1| - v^2 \le 0, \: - v^1 + |v^2| \le 0 \big\}
  = \big\{ (t, t)^T \in \mathbb{R}^2 \bigm| t \ge 0 \big\}, \\
  K_4 &= \Big\{ v \in \mathbb{R}^2 \Bigm| s(\underline{\partial} f(0) + y_-^*, v) \le 0, \:
  s(- x_-^* - \overline{\partial} f(0), v) \le 0 \Big\} \\
  &= \big\{ v \in \mathbb{R}^2 \bigm| |v^1| - v^2 \le 0, \: v^1 + |v^2| \le 0 \big\}
  = \big\{ (-t, t)^T \in \mathbb{R}^2 \bigm| t \ge 0 \big\}.
\end{align*}
One can verify that $T_M(0) = \{ v \in \mathbb{R}^2 \mid |v^1| - |v^2| = 0 \} = \cup_{i = 1}^4 K_i$.
\end{example}

\begin{remark}
The main results of this section can be easily rewritten in terms of \textit{upper convex and lower concave
approximations} of the directional derivatives of the functions $f_i$ and $g_j$ and thus extended to the case when
the functions $f_i$ and $g_j$ are just directionally differentiable (but not necessarily quasidifferentiable). Recall
that a continuous sublinear function $p \colon X \to \mathbb{R}$ is called an upper convex approximation (u.c.a.) of a
positively homogeneous function $h \colon X \to \mathbb{R}$, if $p(v) \ge h(v)$ for all $v \in X$, while a continuous
superlinear function $q \colon X \to \mathbb{R}$ is called a lower concave approximation (l.c.a.) of $h$, 
if $q(v) \le h(v)$ for all $v \in X$. Note that if a function $f$ is quasidifferentiable at a point $x$, then by the
definition of quasidifferential \eqref{eq:QuasidiffDef} for any  $y^* \in \overline{\partial} f(x)$ the function 
$p(\cdot) = s(\underline{\partial} f(x) + y^*, \cdot)$ is an u.c.a. of $f'(x, \cdot)$, while 
for any $x^* \in  \underline{\partial} f(x)$ the function $q(\cdot) = -s(- x^* - \overline{\partial} f(x), \cdot)$ is a
l.c.a. of $f'(x, \cdot)$. However, a function need not be quasidifferentiable to admit upper convex and lower concave
approximations of its directional derivative (see \cite{DemRub_book,Quasidifferentiability_book} for more details).

Suppose that the functions $f_i$ and $g_j$ are directionally differentiable at a point $\overline{x}$. Let $p_i$ be
u.c.a. of $f'_i(\overline{x}, \cdot)$, $q_i$ be l.c.a. of $f'_i(\overline{x}, \cdot)$, $i \in I$, and $u_j$ be u.c.a. of
$g'_j(\overline{x}, \cdot)$, $j \in J(\overline{x})$. Then assumption~\ref{Assumpt_ImplicitMFCQ1} of
Theorem~\ref{thrm:ContingConeToQuasidiffSet} can be rewritten as follows: there exists $v_i \in X$ such that 
$p_i(v_i) < 0$ and for any $k \ne i$ one has $p_k(v_i) \le 0$ and $q_k(v_i) \ge 0$.
Assumptions~\ref{Assumpt_ImplicitMFCQ2} and \ref{Assumpt_ImplicitMFCQ3} of this theorem can be rewritten in a similar
way. Then making necessary changes in the formulation of Theorem~\ref{thrm:ContingConeToQuasidiffSet} and almost
literally repeating its proof one can verify the validity of following inclusion:
$$
  \Big\{ v \in X \Bigm| p_i(v) \le 0, \: q_i(v) \ge 0 \quad \forall i \in I, 
  u_j(v) \le 0 \quad \forall j \in J(\overline{x}) \Big\} \subseteq T_M(\overline{x}).
$$
Optimality conditions from the following section can also be rewritten in terms of u.c.a. of the objective function and
inequality constraints and u.c.a. and l.c.a. of equality constraints. We leave the details to the interested reader.
\end{remark}

\section{Optimality Conditions for Quasidifferentiable Programming Problems}
\label{Section_OptimalityConditions}

In this section we derive the strongest existing necessary optimality conditions for nonsmooth nonlinear programming
problems with quasidifferentiable objective function and constraints under less restrictive assumptions than in
previous studies. Our derivation of these optimality conditions is based on the description of convex subcones of 
the contingent cone given in Theorem~\ref{thrm:ContingConeToQuasidiffSet}.

Consider the following optimization problem
\begin{equation} \label{QuasidiffProgramProblem}
  \min \: f_0(x) \quad \text{s.t.} \quad f_i(x) = 0 \quad \forall i \in I, \quad
  g_j(x) \le 0 \quad \forall j \in J,
\end{equation}
where $f_0, f_i, g_j \colon X \to \mathbb{R}$ are given functions, $I = \{ 1, \ldots, m \}$, and
$J = \{ 1, \ldots, l \}$. Recall that $J(x) = \{ j \in J \mid g_j(x) = 0 \}$.

\begin{theorem} \label{thrm:QuasidiffProgram_OptCond}
Let $\overline{x}$ be a locally optimal solution of problem \eqref{QuasidiffProgramProblem} and the following
assumptions be valid:
\begin{enumerate}
\item{$f_0$ is quasidifferentiable and Hadamard d.d. at $\overline{x}$;}

\item{the functions $f_i$, $i \in I$, are continuous in a neighbourhood of $\overline{x}$ and quasidifferentiable
at $\overline{x}$ uniformly along finite dimensional spaces;
}

\item{the functions $g_j$, $j \notin J(\overline{x})$, are u.s.c. and quasidifferentiable at $\overline{x}$, while the
functions $g_j$, $j \in J(\overline{x})$ are quasidifferentiable at $\overline{x}$ uniformly along finite dimensional
spaces;}

\item{vectors $x_i^* \in \underline{\partial} f_i(\overline{x})$, $y_i^* \in \overline{\partial} f_i(\overline{x})$, 
$i \in I$, and $z_j^* \in \overline{\partial} g_j(\overline{x})$, $j \in J(\overline{x})$, satisfy assumptions
\ref{Assumpt_ImplicitMFCQ1}--\ref{Assumpt_ImplicitMFCQ3} of Theorem~\ref{thrm:ContingConeToQuasidiffSet}.
} 
\end{enumerate}
Then for any $y_0^* \in \overline{\partial} f_0(\overline{x})$ and $z_j^* \in \overline{\partial} g_j(\overline{x})$, 
$j \notin J(\overline{x})$, there exist $\lambda_j \ge 0$, $j \in J$, such that $\lambda_j g_j(\overline{x}) = 0$ for
any $j \in J$ and
\begin{equation} \label{QuasidiffProg_OptCond_WithCone}
  0 \in \underline{\partial} f_0(\overline{x}) + y_0^* 
  + \sum_{j = 1}^l \lambda_j \big( \underline{\partial} g_j(\overline{x}) + z_j^* \big) 
  + \cl^* \cone\big\{ C_i \mid i \in I \big\},
\end{equation}
where $C_i = (\underline{\partial} f_i(\overline{x}) + y_i^*) \cup (- x_i^* - \overline{\partial} f_i(\overline{x}))$.
\end{theorem}

\begin{proof}
With the use of the definitions of contingent cone and Hadamard directional derivative one can easily verify that 
the local optimality of the point $\overline{x}$ implies that $f_0'(\overline{x}, v) \ge 0$ for any 
$v \in T_M(\overline{x})$, where $M$ is the feasible region of problem \eqref{QuasidiffProgramProblem}
(see~\eqref{QuasidiffSet}). Hence, in particular, $f_0'(\overline{x}, v) \ge 0$ for any $v \in K$, where
\begin{equation} \label{ConvexSubconeContingCone}
\begin{split}
  K = \Big\{ v \in X \Bigm| s\big( \underline{\partial} f_i(\overline{x}) + y_i^*, v \big) \le 0, \enspace
  &s\big( - x_i^* - \overline{\partial} f_i(\overline{x}), v \big) \le 0, \enspace i \in I, \\
  &s\big( \underline{\partial} g_j(\overline{x}) + z_j^*, v \big) \le 0, \enspace j \in J(\overline{x}) \Big\},
\end{split}
\end{equation}
since by Theorem~\ref{thrm:ContingConeToQuasidiffSet} one has $K \subseteq T_M(\overline{x})$. 

Choose any $y_0^* \in \overline{\partial} f_0(\overline{x})$. By the definition of quasidifferential one has
$$
  p(v) = s\big( \underline{\partial} f_0(\overline{x}) + y_0^*, v \big) \ge f_0'(\overline{x}, v)
  \quad \forall v \in X.
$$
Therefore $p(v) \ge 0$ for any $v \in K$, which, as is readily seen, implies that $0$ is a globally optimal solution
of the convex programming problem
\begin{equation} \label{AuxConvexProblem}
  \min \: p(v) \quad \text{s.t.} \quad q_j(v) \le 0 \quad \forall j \in J(\overline{x}), \quad v \in H,
\end{equation}
where $q_j(v) = s(\underline{\partial} g_j(\overline{x}) + z_j^*, v)$ and
$$
  H = \Big\{ v \in X \Bigm| s\big( \underline{\partial} f_i(\overline{x}) + y_i^*, v \big) \le 0, \enspace
  s\big( - x_i^* - \overline{\partial} f_i(\overline{x}), v \big) \le 0, \enspace i \in I \Big\}.
$$
Note that the cone $H$ is obviously closed and convex. By assumption~\ref{Assumpt_ImplicitMFCQ3} of
Theorem~\ref{thrm:ContingConeToQuasidiffSet} there exists $v_0 \in H$ such that $q_j(v_0) < 0$ for any 
$j \in J(\overline{x})$, i.e. Slater's condition for problem~\eqref{AuxConvexProblem} holds true. Consequently, applying
the necessary and sufficient optimality conditions for convex programming problems (see, e.g.
\cite[Theorem~$1.1.2'$]{IoffeTihomirov}) one obtains that there exists $\lambda_j \ge 0$, $j \in J(\overline{x})$, such
that
\begin{equation} \label{OptCond_AuxConvexProblem}
  0 \in \partial p(0) + \sum_{j \in J(\overline{x})} \lambda_j \partial q_j(0) + H^o
\end{equation}
where $H^o = \{ x^* \in X^* \mid \langle x^*, v \rangle \le 0 \: \forall v \in H \}$ is the polar
cone of $H$ and $\partial$ is the subdifferential in the sense of convex analysis. We claim that
\begin{equation} \label{PolarConeViaSubdiff}
  H^o = \cl^* \cone\big\{ C_i \mid i \in I \big\},
\end{equation}
where $C_i = (\underline{\partial} f_i(\overline{x}) + y_i^*) \cup (- x_i^* - \overline{\partial} f_i(\overline{x}))$.
Indeed, the inclusion ``$\supseteq$'' follows directly from the definition of $H$. Arguing by reductio ad
absurdum, suppose that the opposite inclusion does not hold true, i.e. that there exists $x^* \in H^o$ such
that $x^* \notin \cl^* \cone\{ C_i \mid i \in I \}$. Then applying the separation theorem in the space $X^*$
equipped with the weak${}^*$ topology one gets that there exists $v \in X$ such that $\langle x^*, v \rangle > 0$, while
$\langle y^*, v \rangle \le 0$ for any $y^* \in \cl^* \cone\{ C_i \mid i \in I \}$. From the second inequality
it follows that $v \in H$ by the definition of $H$, which is impossible, since 
$x^* \in H^o$ and $\langle x^*, v \rangle > 0$. Thus, \eqref{PolarConeViaSubdiff} holds true. Consequently,
computing the subdifferentials $\partial p(0)$ and $\partial q_j(0)$ with the use of the theorem on 
the subdifferential of the supremum of a family of convex functions (see, e.g. \cite[Theorem~4.2.3]{IoffeTihomirov}),
setting $\lambda_j = 0$ for any $j \notin J(\overline{x})$, and applying \eqref{OptCond_AuxConvexProblem} one obtains
that optimality condition \eqref{QuasidiffProg_OptCond_WithCone} holds true.
\end{proof}

\begin{corollary} \label{crlr:QuasidiffProg_LagrangeMultipliers}
Let all assumptions of the theorem above be valid and suppose that the set $\cone\{ C_i \mid i \in I \}$ is weak${}^*$
closed. Then for any $y_0^* \in \overline{\partial} f_0(\overline{x})$ and 
$z_j^* \in \overline{\partial} g_j(\overline{x})$, $j \notin J(\overline{x})$, there exist $\underline{\mu}_i \ge 0$,
$\overline{\mu}_i \ge 0$, $i \in I$, and $\lambda_j \ge 0$, $j \in J$, such that $\lambda_j g_j(\overline{x}) = 0$ for
any $j \in J$ and
\begin{equation} \label{QuasidiffProg_LagrageMultipliers}
\begin{split}
  0 \in \underline{\partial} f_0(\overline{x}) + y_0^* 
  &+ \sum_{i = 1}^m \underline{\mu}_j \big( \underline{\partial} f_i(\overline{x}) + y_i^* \big) \\
  &- \sum_{i = 1}^m \overline{\mu}_j \big( x_i^* + \overline{\partial} f_i(\overline{x}) \big)
  + \sum_{j = 1}^l \lambda_j \big( \underline{\partial} g_j(\overline{x}) + z_j^* \big).
\end{split}
\end{equation}
\end{corollary}

\begin{proof}
By Theorem~\ref{thrm:QuasidiffProgram_OptCond} there exist $\lambda_j \ge 0$, $j \in J$, and
$x^* \in \underline{\partial} f_0(\overline{x}) + y_0^* + 
\sum_{j \in J} \lambda_j (\underline{\partial} g_j(\overline{x}) + z_j^*)$ such that 
$- x^* \in \cone\{ C_i \mid i \in I \}$ and $\lambda_j g_j(\overline{x}) = 0$ for any $j \in J$. From the definitions of
conic hull and the sets $C_i$ it follows that there exist $\underline{\mu}_i \ge 0$ and 
$\overline{\mu}_i \ge 0$, $i \in I$, such that
$$
  - x^* \in \sum_{i = 1}^m \underline{\mu}_j \big( \underline{\partial} f_i(\overline{x}) + y_i^* \big)
  - \sum_{i = 1}^m \overline{\mu}_j \big( x_i^* + \overline{\partial} f_i(\overline{x}) \big),
$$
i.e. \eqref{QuasidiffProg_LagrageMultipliers} holds true.
\end{proof}

\begin{remark}
Note that each equality constraint $f_i(x) = 0$ enters optimality condition \eqref{QuasidiffProg_LagrageMultipliers}
\textit{twice}, as two inequality constraints, namely, $f_i(x) \ge 0$ and $f_i(x) \le 0$, which seems to be a specific
feature of optimality conditions in terms of quasidifferentials that is connected to the fact that in
quasidifferentiable programming constraints $g(x) \le 0$ and $h(x) \ge 0$ enter optimality conditions differently
(usually, only the sign of the corresponding multiplier changes). Both $\underline{\mu}_i$ and $\overline{\mu}_i$ in
\eqref{QuasidiffProg_LagrageMultipliers} can be viewed as multipliers corresponding to the equality constraint 
$f_i(x) = 0$. Thus, loosely speaking, one can say that there are \textit{two} multipliers $\underline{\mu}_i$ 
and $\overline{\mu}_i$ corresponding to each equality constraint $f_i(x) = 0$. Finally, let us note that Lagrange
multipliers $\lambda_i$, $\underline{\mu}_i$, and $\overline{\mu}_i$ obviously depend on the choice of the vectors
$x_i^*$, $y_i^*$ and $z_j^*$ from the corresponding quasidifferentials of the objective function and constraints and
\textit{cannot} be chosen independently of these vectors in the general case
(cf.~\cite{LudererRosiger90,Luderer92,WangMortensen}).
\end{remark}

Let us point out a simple sufficient condition for the weak${}^*$ closedness of the convex conic hull 
$\cone\{ C_i \mid i \in I \}$ from the corollary above, which is satisfied in almost all finite dimensional
applications. In the finite dimensional case the subdifferentials $\underline{\partial} f_i(\overline{x})$ and the
superdifferentials $\overline{\partial} f_i(\overline{x})$ are usually polytopes (i.e. convex hulls of a finite number
of points). Clearly, one can replace these polytopes in the definition of
$\cone\{ C_i \mid i \in I \}$ with their extreme points, i.e. 
\begin{multline*}
  \cone\{ C_i \mid i \in I \} 
  = \cone\Big\{ x^* \in X^* \Bigm| \\
  x^* \in \extreme\big( \underline{\partial} f_i(\overline{x}) + y_i^* \big) \cup
  \extreme \big( - x_i^* - \overline{\partial} f_i(\overline{x}) \big), i \in I \Big\},
\end{multline*}
where $\extreme A$ is the set of extreme points of a convex set $A$. By the definition of polytope the sets
$\extreme( \underline{\partial} f_i(\overline{x}) + y_i^* )$ and
$\extreme( - x_i^* - \overline{\partial} f_i(\overline{x}) )$ are finite. Thus, if 
the sets $\underline{\partial} f_i(\overline{x})$ and $\overline{\partial} f_i(\overline{x})$, $i \in I$, are polytopes,
then the cone $K = \cone\{ C_i \mid i \in I \}$ is finitely generated and, as is well-known, weak${}^*$ closed (see,
e.g. \cite[Proposition~2.41]{BonnansShapiro}).

In the case when there are no equality constraints one can obtain a slightly stronger result than the one given in
Theorem~\ref{thrm:QuasidiffProgram_OptCond}. For any convex set $A$ denote by
$N_A(x) = \{ x^* \in X^* \mid \langle x^*, y - x \rangle \le 0 \: \forall y \in A \}$ \textit{the normal cone} to the
set $A$ at a point $x \in A$ in the sense of convex analysis.

\begin{theorem} \label{thrm:InequalConstr_OptCond}
Let $\overline{x} \in X$ be a locally optimal solution of the problem
$$
  \min \: f_0(x) \quad \text{s.t.} \quad g_j(x) \le 0 \quad \forall j \in J, \quad x \in A,
$$
where $A \subset X$ is a closed convex set. 
Suppose that the functions $f_0$ and $g_j$, $j \in J$ are quasidifferentiable at $\overline{x}$, and let
$z_j^* \in \overline{\partial} g_j(\overline{x})$, $j \in J$, be such that the following constraint qualification holds
true:
\begin{equation} \label{CQ_QuasidiffInequal}
  0 \notin \co\{ \underline{\partial} g_j(\overline{x}) + z_j^* \mid j \in J(\overline{x}) \} + N_A(\overline{x}).
\end{equation}
Then for any $y_0^* \in \overline{\partial} f_0(\overline{x})$ there exist $\lambda_j \ge 0$, $j \in J$, such that 
\begin{equation} \label{QuasidiffProg_InequalConstr_LMRule}
  0 \in \underline{\partial} f_0(\overline{x}) + y_0^* 
  + \sum_{j = 1}^l \lambda_j \big( \underline{\partial} g_j(\overline{x}) + z_j^* \big) + N_A(\overline{x}),
  \quad \lambda_j g_j(\overline{x}) = 0 \quad \forall j \in J.
\end{equation}
\end{theorem}

\begin{proof}
Define $h(x) = \max_{j \in J} \{ f_0(x) - f_0(\overline{x}), g_j(x) \}$. Applying standard calculus rules for
directional derivatives \cite{DemRub_book} one can check that the function $h$ is d.d. at $\overline{x}$ and
\begin{equation} \label{DirectDeriv_MaxFunc}
  h'(\overline{x}, v) = \max_{j \in J(\overline{x})} \{ f'_0(\overline{x}, v), g_j'(\overline{x}, v) \}
  \quad \forall v \in X.
\end{equation}
It is readily seen that $\overline{x}$ is a point of local minimum of the function $h$ on the set $A - \overline{x}$.
Therefore, $h'(\overline{x}, v) \ge 0$ for any $v \in A - \overline{x}$ due to the convexity of the set $A$. 

Fix any $y_0^* \in \overline{\partial} f_0(\overline{x})$. By the definition of quasidifferential one has
$f_0'(x, v) \le s(\underline{\partial} f_0(\overline{x}) + y_0^*, v)$ and
$g_j'(x, v) \le s(\underline{\partial} g_j(\overline{x}) + z_j^*, v)$ for any $v \in X$ and $j \in J(\overline{x})$.
Hence with the use of \eqref{DirectDeriv_MaxFunc} one gets that
$$
  \eta(v) = \max_{j \in J(\overline{x})} \big\{ s(\underline{\partial} f_0(\overline{x}) + y_0^*, v),
  s(\underline{\partial} g_j(\overline{x}) + z_j^*, v) \big\} \ge h'(\overline{x}, v) \ge 0
$$
for any $v \in A - \overline{x}$, i.e. $0$ is a point of global minimum of the convex function $\eta$ on the set 
$A - \overline{x}$. Therefore, $0 \in \partial \eta(0) + N_A(\overline{x})$ (see, e.g.
\cite[Theorem~1.1.2']{IoffeTihomirov}). Applying the theorem on
the subdifferential of the supremum of a family of convex functions \cite[Theorem~4.2.3]{IoffeTihomirov} one gets that 
$\partial \eta(0) = \co_{j \in J(\overline{x})}
\{ \underline{\partial} f_0(\overline{x}) + y_0^*, \underline{\partial} g_j(\overline{x}) + z_j^* \}$, which implies
that there exist $\alpha_0 \ge 0$ and $\alpha_j \ge 0$, $j \in J(\overline{x})$, such that 
$\alpha_0 + \sum_{j \in J(\overline{x})} \alpha_j = 1$, and
$$
  0 \in \alpha_0 \big( \underline{\partial} f_0(\overline{x}) + y_0^* \big) 
  + \sum_{j \in J(\overline{x})} \alpha_j \big( \underline{\partial} g_j(\overline{x}) + z_j^* \big) 
  + N_A(\overline{x}).
$$
Note that if $\alpha_0 = 0$, then 
$0 \in \co\{ \underline{\partial} g_j(\overline{x}) + z_j^* \mid j \in J(\overline{x}) \} + N_A(\overline{x})$, which
contradicts \eqref{CQ_QuasidiffInequal}. Thus, $\alpha_0 \ne 0$. Hence dividing the inclusion above by $\alpha_0$ one
obtains that \eqref{QuasidiffProg_InequalConstr_LMRule} holds true with $\lambda_j = \alpha_j / \alpha_0$ for any 
$j \in J(\overline{x})$ and $\lambda_j = 0$ for any $j \notin J(\overline{x})$.
\end{proof}

Let us present two simple examples that illustrate Theorems~\ref{thrm:QuasidiffProgram_OptCond} and
\ref{thrm:InequalConstr_OptCond} and, at the same time, demonstrate that optimality conditions in terms of
quasidifferentials are sometimes better than optimality conditions in terms of various subdifferentials. In these
examples we consider optimization problems without equality constraints. A similar example of an equality constrained
problem is given in \cite{Dolgopolik_MetricReg}.

\begin{example}
Firstly we analyze a problem with a degenerate constraint. Let $X = \mathbb{R}$, and consider the following
optimization problem:
\begin{equation} \label{Example_DegenInequalConstr}
  \min \: f_0(x) = x \quad \text{s.t.} \quad g(x) = \min\{ x, x^3 \} \le 0.
\end{equation}
The point $\overline{x} = 0$ is obviously not a locally optimal solution of this problem, since the set $(- \infty, 0]$
is a feasible region of this problem. However, let us check that optimality conditions in terms of various
subdifferentials hold true at $\overline{x}$.

Denote by $L(x, \lambda_0, \lambda) = \lambda_0 f_0(x) + \lambda g(x)$ the Lagrangian for problem
\eqref{Example_DegenInequalConstr}. It is easy to see that $\partial_{Cl} L(\overline{x}, 0, \lambda) = [0, \lambda]$
for any $\lambda > 0$, where $\partial_{Cl}$ is the Clarke subdifferential. Thus, 
$0 \in \partial_{Cl} L(\overline{x}, 0, \lambda)$, i.e. the optimality conditions in terms of the Clarke subdifferential
\cite[Theorem~6.1.1]{Clarke} are satisfied at $\overline{x}$.

Let us now consider the Michel-Penot subdifferential \cite{Ioffe93}. Fix any $\lambda > 0$. For any $v \in \mathbb{R}$
the Michel-Penot directional derivative of the function $L(\cdot, 0, \lambda)$ at $\overline{x}$ has the form
\begin{align*}
  d_{MP} L(\cdot, 0, \lambda)[\overline{x}, v] 
  &= \sup_{e \in \mathbb{R}} \limsup_{t \to +0} 
  \frac{L(\overline{x} + t (v + e), 0, \lambda) - L(\overline{x} + t e, 0, \lambda)}{t} \\
  &= \sup_{e \in \mathbb{R}} \lambda \big( \min\big\{ v + e, 0 \big\} - \min\{ e, 0 \} \big) 
  = \lambda \max\{ 0, v \}.
\end{align*}
Consequently, $\partial_{MP} L(\cdot, 0, \lambda)(\overline{x}) = [0, \lambda]$, where $\partial_{MP}$ is the
Michel-Penot sub\-differential of $L(\cdot, 0, \lambda)$ at $\overline{x}$. Thus, 
$0 \in \partial_{MP} L(\cdot, 0, \lambda)(\overline{x})$, i.e. the optimality conditions in terms of the Michel-Penot
subdifferential \cite{Ioffe93} are satisfied at $\overline{x}$.

Let us now turn to approximate (Ioffe) subdifferential \cite{Ioffe84,Ioffe2012,Penot_book}. Observe that for any 
$x \in (0, 1)$ one has $L(x, 0, \lambda) = \lambda x^3$. Therefore, for any such $x$ one has
$\partial^-L(\cdot, 0, \lambda)(x) = 3 \lambda x^2$, where $\partial^-$ is the Dini subdifferential. Hence for the
Ioffe subdifferential one has 
$0 \in \partial_a L(\cdot, 0, \lambda)(\overline{x}) = \limsup_{x \to \overline{x}} \partial^-L(\cdot, 0, \lambda)(x)$,
where $\limsup$ is the outer limit. Thus, the optimality conditions in terms of the Ioffe subdifferential
\cite[Proposition~12]{Ioffe84} are satisfied at $\overline{x}$ for any $\lambda > 0$. 

Denote by $\partial_M$ the Mordukhovich basic subdifferential \cite{Mordukhovich_I,Mordukhovich_II}. With the use of
the representation of this subdifferential as the limiting Fr\'{e}chet subdifferential
\cite[Theorem~1.89]{Mordukhovich_I} one can easily check that 
$\partial_M (\lambda g)(\overline{x}) = \{ 0, \lambda \}$. Consequently, 
$- \lambda_0 \nabla f_0(\overline{x}) \in \partial_M (\lambda g)(\overline{x})$ for $\lambda_0 = 0$ and any 
$\lambda > 0$, i.e. the optimality conditions in terms of the Mordukhovich subdifferential 
\cite[Theorem~5.19]{Mordukhovich_II} are satisfied at $\overline{x}$ as well.

Let us now consider the Jeyakumar-Luc subdifferential \cite{WangJeyakumar}, which we denote by $\partial_{JL}$. One can
check that $\partial_{JL} g(\overline{x}) = \{ 0, 1 \}$ is the smallest Jeykumar-Luc subdifferential of $g$ at
$\overline{x}$. For any $\lambda > 0$ and $\lambda_0 = 0$ one has 
$0 \in \lambda_0 \nabla f_0(\overline{x}) + \lambda \co \partial_{JL} g(\overline{x})$. Thus, the optimality conditions
in terms of the Jeyakumar-Luc subdifferential \cite[Corollary~3.4]{WangJeyakumar} are satisfied at $\overline{x}$.

Let us finally check that optimality conditions in terms of quasidifferentials
(Theorem~\ref{thrm:InequalConstr_OptCond}), in contrast to optimality conditions in terms of subdifferentials,
detect the nonoptimality of the point $\overline{x} = 0$. Indeed, the function $g$ is obviously quasidifferentiable at
$\overline{x}$ and one can define $\mathscr{D} g(\overline{x}) = [ \{ 0 \}, [0, 1] ]$. Note that for 
$z^* = 1 \in \overline{\partial} g(\overline{x})$ one obviously has 
$0 \notin \underline{\partial} g(\overline{x}) + z^*$, i.e. the assumptions of
Theorem~\ref{thrm:InequalConstr_OptCond} are satisfied. Therefore, if $\overline{x}$ is a locally optimal
solution of problem \eqref{Example_DegenInequalConstr}, then by Theorem~\ref{thrm:InequalConstr_OptCond}
there exists $\lambda \ge 0$ such that 
$$
  0 \in \nabla f_0(\overline{x}) + \lambda \big( \underline{\partial} g(\overline{x}) + z^* \big) 
  = 1 + \lambda 1 = 1 + \lambda,
$$
which is clearly impossible. Thus, one can conclude that the point $\overline{x}$ is nonoptimal. 
\end{example}

\begin{remark} \label{remark:InEqConstr_Diff}
Various optimality conditions and constraint qualifications for quasidifferentiable programming problems with
inequality constraints were analyzed in
\cite{DemyanovPolyakova80,Ward91,LudererRosiger91,KuntzScholtes,KuntzScholtes92,Dolgopolik_MetricReg}. One can
check that none of the constraint qualifications from these papers are satisfied for 
problem~\eqref{Example_DegenInequalConstr} at the point  $\overline{x} = 0$. Moreover, the so-called 
\textit{nondegeneracy condition}
$$
  \cl \{ v \in X \mid g'(\overline{x}, v) < 0 \} = \{ v \in X \mid g'(\overline{x}, v) \le 0 \}
$$
does not hold true at $\overline{x}$ either. Thus, it seems that in the case of quasidifferentiable programming problems
constraint qualifications must depend on individual elements of quasidifferentials just like Lagrange multipliers in
quasidifferentiable programming depend on individual elements of quasidifferentials. To the best of the author's
knowledge (and much to the author's surprise), such constraint qualifications have never been analyzed before.
\end{remark}

\begin{example}
Let us also consider a nondegenerate problem. Let $X = \mathbb{R}^2$ and consider the following
optimization problem:
\begin{equation} \label{Example_DC_InequalConstr}
  \min \: f_0(x) = |x^1| - |x^2| \quad \text{s.t.} \quad g(x) = - x^1 + x^2 \le 0.
\end{equation}
The point $\overline{x} = 0$ is not a locally optimal solution of this problem, since for any $t > 0$ the point 
$x(t) = (t, - 2t)$ is feasible for this problem and the inequality $f_0(x(t)) = - t < 0 = f_0(\overline{x})$ holds true.

Denote by $L(x, \lambda) = f_0(x) + \lambda g(x)$ the Lagrangian for problem \eqref{Example_DC_InequalConstr}.
One can easily check that
$$
  \partial_{MP} L(\cdot, \lambda)(\overline{x}) = \partial_{Cl} L(\cdot, \lambda)(\overline{x})
  = \co\left\{ \begin{pmatrix} 1 \\ 1 \end{pmatrix}, \begin{pmatrix} 1 \\ -1 \end{pmatrix},
  \begin{pmatrix} - 1 \\ 1 \end{pmatrix}, \begin{pmatrix} - 1 \\ -1 \end{pmatrix} \right\} 
  + \lambda \begin{pmatrix} - 1 \\ 1 \end{pmatrix}.
$$
Therefore optimality conditions in terms of the Michel-Penot and Clarke subdifferentials are satisfied at
$\overline{x}$ for any $\lambda \in [0, 1]$. By \cite[Example~2.1]{WangJeyakumar} one can set 
$\partial_{JL} f_0(\overline{x}) = \{ (1, -1)^T, (-1, 1)^T \}$, which implies that 
$0 \in \co \partial_{JL} f_0(\overline{x}) + \lambda \nabla g(\overline{x})$ for $\lambda = 1$, i.e. the optimality
conditions in terms of the Jeyakumar-Luc subdifferential are satisfied at $\overline{x}$ as well.

By \cite[p.~92--93]{Mordukhovich_I} one has
$\partial_M f_0(\overline{x}) = \co\{ (\pm 1, -1)^T \} \cup \co\{ (\pm 1, 1)^T \}$. Therefore
$0 \in \partial_M f_0(\overline{x}) + \lambda \nabla g(\overline{x})$ for $\lambda = 1$, i.e. the optimality conditions
in terms of the Mordukhovich basic subdifferential are satisfied at $\overline{x}$. Finally, for any 
$x \in \mathbb{R}^2$ such that $x^1, x^2 > 0$ one has $L(x, 1) = 0$, which implies that
$\partial^- L(\cdot, 1)(x) = \{ 0 \}$ for any such $x$. Hence 
$0 \in \partial_a L(\cdot, 1)(\overline{x}) = \limsup_{x \to \overline{x}} \partial^- L(x, 1)$, i.e. the optimality
conditions in term of Ioffe's approximate subdifferential are satisfied at $\overline{x}$ as well.

Let us now consider optimality conditions in terms of quasidifferentials. The function $f_0$ is 
quasidifferentiable at $\overline{x}$ and one can define
$$
  \underline{\partial} f_0(\overline{x}) = \co \left\{ \begin{pmatrix} 1 \\ 0 \end{pmatrix},
  \begin{pmatrix} - 1 \\ 0 \end{pmatrix} \right\}, \quad
  \overline{\partial} f_0(\overline{x}) = \co \left\{ \begin{pmatrix} 0 \\ 1 \end{pmatrix},
  \begin{pmatrix} 0 \\ -1 \end{pmatrix} \right\}. 
$$
For $y_0^* = (0, 1)^T \in \overline{\partial} f_0(\overline{x})$ one has 
$\underline{\partial} f_0(\overline{x}) + y_0^* = \co\{ (1, 1)^T, (-1, 1)^T \}$. Therefore, 
$0 \notin \underline{\partial} f_0(\overline{x}) + y_0^* + \lambda \nabla g(\overline{x})$ for any $\lambda \ge 0$.
Consequently, the optimality conditions from Theorem~\ref{thrm:InequalConstr_OptCond} are not satisfied at
$\overline{x}$, and one can conclude that the point $\overline{x}$ is nonoptimal, since the constraint qualification
$\nabla g(\overline{x}) = (-1, 1)^T \ne 0$ holds true at $\overline{x}$. Thus, unlike optimality conditions in terms of
subdifferentials, the optimality conditions in terms of quasidifferentials are able to detect the nonoptimality of this
point.
\end{example}

\begin{remark}
Let $X = \mathbb{R}^n$ and ``$\partial$'' be any subdifferential mapping that satisfies the following assumption: if
a function $f \colon \mathbb{R}^n \to \mathbb{R}$ is continuously differentiable at a sequence of points 
$\{ x_n \} \subset \mathbb{R}^n$ converging to some $x \in \mathbb{R}^n$ and there exists the limit 
$v = \lim_{n \to \infty} \nabla f(x_n)$, then $v \in \partial f(x)$. Then in the previous example one has 
$0 \in \partial L(\cdot, 1)(\overline{x})$ and $0 \in \partial f_0(\overline{x}) + \nabla g(\overline{x})$ due to our
assumption on ``$\partial$'' and the fact that for any $x \in \mathbb{R}^2$ such that $x^1, x^2 > 0$ one has 
$L(x, 1) = 0$, i.e. $\nabla_x L(x, 1) = 0$, and $\nabla f_0(x) = (1, -1)^T$. Thus, roughly speaking, no outer
semicontinuous/limiting subdifferential can detect the nonoptimality of the point $\overline{x}$ in the previous
example.
\end{remark}

\section{A comparison of constraint qualifications}
\label{Section_CQ_Comparison}

As was pointed out in the introduction, numerous papers have been devoted to analysis of constraint qualifications and
optimality conditions for nonsmooth quasidifferentiable programming problems with equality and inequality constraints.
Therefore, it is necessary to point out the difference between the main results of this paper and previous studies. 

Let functions $f_0, g \colon X \to \mathbb{R}$ be quasidifferentiable at a point $\overline{x}$ satisfying the
inequality $g(\overline{x}) \le 0$. Consider the following optimization problem:
$$
  \min \: f_0(x) \quad \text{subject to} \quad g(x) \le 0.
$$
A detailed analysis of constraint qualifications in terms of quasidifferentials for this problem was presented in
\cite{KuntzScholtes}. The most widely used constraint qualification for such problems is the 
\textit{nondegenracy condition}  
$$
  \cl \{ v \in X \mid g'(\overline{x}, v) < 0 \} = \{ v \in X \mid g'(\overline{x}, v) \le 0 \},
$$
which was first utilized in quasidifferentiable optimization by Demyanov and Polyakova \cite{DemyanovPolyakova80}. As
was pointed out in \cite{KuntzScholtes}, ``for a given problem it is usually hard if not impossible to verify the
nondegeneracy condition''. Therefore, different constraint qualifications are needed. In \cite{KuntzScholtes} it was
shown that the strongest constraint qualification among existing ones in terms of quasidifferentials is the assumption
that the pair $(\underline{\partial} g(\overline{x}), - \overline{\partial} g(\overline{x}))$ is \textit{in general
position}, in the sense that the validity of \textit{all} other existing constraints qualifications implies that the
pair $(\underline{\partial} g(\overline{x}), - \overline{\partial} g(\overline{x}))$ is in general position. This
assumption was introduced by Rubinov, and it is invariant with respect to the choice of quasidifferential (see, e.g.
\cite{DemRub_book,Quasidifferentiability_book}). Recall that a pair $[A, B]$ of weak${}^*$ compact convex subsets of
$X^*$ is said to be in \textit{general position}, if for any $v \in X$  \textit{the max-face} 
$$
  \Delta(v \mid B) = \{ y^* \in B \mid \langle y^*, v \rangle = s(B, v) \}
$$
is \textit{not} contained in the max-face 
$$
  \Delta(v \mid A) = \{ x^* \in A \mid \langle x^*, v \rangle = s(A, v) \}. 
$$
If the pair $(\underline{\partial} g(\overline{x}), - \overline{\partial} g(\overline{x}))$ is in general position, then
by definition the max-face 
$\Delta(0 \mid - \overline{\partial} g(\overline{x})) = - \overline{\partial} g(\overline{x})$ corresponding to the
vector $v = 0$ is not contained in the max-face
$\Delta(0 \mid \underline{\partial} g(\overline{x})) = \underline{\partial} g(\overline{x})$. Therefore, there exists
$z^* \in \overline{\partial} g(\overline{x})$ such that $0 \notin \underline{\partial} g(\overline{x}) + z^*$, that is,
constraint qualification \eqref{CQ_QuasidiffInequal} from Theorem~\ref{thrm:InequalConstr_OptCond} is satisfied for
some vectors $z^* \in \overline{\partial} g(\overline{x})$. However, in many particular cases this constraint
qualification can be satisfied, when 
the pair $(\underline{\partial} g(\overline{x}), - \overline{\partial} g(\overline{x}))$ is not in general position.

\begin{example} \label{example:CQ_ne_GeneralPosition}
Let $X = \mathbb{R}^2$ and $g(x) = \max\{ 2 x^1, 2 x^2 \} + \min\{ 0, - x^1 - x^2 \}$. This function is
quasidifferentiable and its quasidifferential at the point $\overline{x} = 0$ has the form
$$
  \underline{\partial} g(0) = \co\left\{ \begin{pmatrix} 2 \\ 0 \end{pmatrix}, 
  \begin{pmatrix} 0 \\ 2 \end{pmatrix} \right\}, \quad
  \overline{\partial} g(0) = \co\left\{ \begin{pmatrix} 0 \\ 0 \end{pmatrix}, 
  \begin{pmatrix} -1 \\ -1 \end{pmatrix} \right\}
$$
For $v = (1, 1)^T$ one has $\Delta(v \mid - \overline{\partial} g(0)) = \{ (1, 1)^T) \}$ and
$\Delta(v \mid \underline{\partial} g(0)) = \underline{\partial} g(0)$,  which implies that the pair 
$(\underline{\partial} g(0), - \overline{\partial} g(0))$ is not in general position, since
$(1, 1)^T \in \underline{\partial} g(0)$. On the other hand, for any 
$z^* \in \overline{\partial} g(0) \setminus \{ (-1, -1)^T \}$ one has
$0 \notin \underline{\partial} g(0) + z^*$.
\end{example}

Let us now consider the equality constrained problem
$$
  \min \: f_0(x) \quad \text{subject to} \quad f_1(x) = 0.
$$
Optimality condition \eqref{QuasidiffProg_OptCond_WithCone} for this problem was first derived by Polyakova
\cite{Polyakova86} under the assumption that $T_M(\overline{x}) = \{ v \in X \mid f'_1(\overline{x}, v) = 0 \}$, where
by definition $M = \{ x \in X \mid f_1(x) = 0 \}$. Furthermore, it was shown in \cite{Polyakova86} that this assumptions
is satisfied, provided $0 \notin [\mathscr{D} f_1(\overline{x})]^+$, i.e. provided q.d.-MFCQ holds at $\overline{x}$.
Note that the constraint qualification from Theorem~\ref{thrm:ContingConeToQuasidiffSet} (see also
Corollary~\ref{crlr:ContingCone_EqualConstr}) that we use is much less restrictive than q.d.-MFCQ. In  particular,
q.d.-MFCQ is not satisfies for the constraint $f_1(x) = |\sin x^1| - |\sin x^2| = 0$ at the point $\overline{x} = 0$,
while the constraint qualification from Theorem~\ref{thrm:ContingConeToQuasidiffSet} is satisfied for many particular
elements of a quasidifferential of $f_1$ (see~Example~\ref{example:CrossEqConstr}). 

In turn, if q.d.-MFCQ is not satisfied, then, unlike the constraint qualification from
Theorem~\ref{thrm:ContingConeToQuasidiffSet}, 
the assumption $T_M(\overline{x}) = \{ v \in X \mid f'_1(\overline{x}, v) = 0 \}$ is hard to verify directly without
employing some additional information about the function $f_1$ apart from its quasidifferential at $\overline{x}$. In
addition, there are cases when this assumption is not satisfied, while the constraint qualification from
Theorem~\ref{thrm:ContingConeToQuasidiffSet} can be applied. 

\begin{example} \label{example:KernelDirectDeriv_ne_ContingCone}
Let $X = \mathbb{R}^2$ and $$
  f_1(x) = \max\{ \sin x^1 + \sin x^2, 0 \} + \min\{ - x^1 - x^2, x^1 \}.
$$
Then for the point $\overline{x} = 0$ 
one has $f_1'(\overline{x}, v) = \max\{ v^1 + v^2, 0 \} + \min\{ - v^1 - v^2, v^1 \}$, which yields
$$
  K = \big\{ v = (v^1, v^2) \in \mathbb{R}^2 \bigm| v^1, v^2 > 0 \big\} \subset 
  \big\{ v \in X \bigm| f'_1(\overline{x}, v) = 0 \big\}. 
$$
However, from the fact that $t > \sin t$ for all $t > 0$ it follows that for any $x \in (0, \pi) \times (0, \pi)$ one
has $f_1(x) = \sin x^1 + \sin x^2 - x^1 - x^2 < 0$ , which implies that $K \cap T_M(\overline{x}) = \emptyset$, that is,
$T_M(\overline{x}) \ne \{ v \in \mathbb{R}^2 \mid f'_1(\overline{x}, v) = 0 \}$. On the other hand, applying 
the standard rules of the quasidifferential calculus (see, e.g. \cite{DemRub_book}) one can check that $f_1$ is
quasidifferentiable at $\overline{x}$ and one 
can define $\mathscr{D} f_1(0) = [\underline{\partial} f_1(0), \overline{\partial} f_1(0)]$ with
$$
  \underline{\partial} f_1(0) = \co\left\{ \begin{pmatrix} 1 \\ 1 \end{pmatrix}, 
  \begin{pmatrix} 0 \\ 0 \end{pmatrix} \right\}, \quad
  \overline{\partial} f_1(0) = \co\left\{ \begin{pmatrix} -1 \\ -1 \end{pmatrix},
  \begin{pmatrix} 1 \\ 0 \end{pmatrix} \right\}.
$$
Observe that for $x^* = (0, 0)^T \in \underline{\partial} f_1(0)$ one has $0 \notin x^* + \overline{\partial} f_1(0)$
and for $y^* = (1, 0)^T \in \overline{\partial} f_1(0)$ one has $0 \notin \underline{\partial} f_1(0) + y^*$, i.e. the
constraint qualification from Theorem~\ref{thrm:ContingConeToQuasidiffSet} holds true at $\overline{x} = 0$
(see Corollary~\ref{crlr:ContingCone_EqualConstr}).
\end{example}

In \cite{DemRub_book,Quasidifferentiability_book} it was shown that the condition
$T_M(\overline{x}) = \{ v \in X \mid f'_1(\overline{x}, v) = 0 \}$ is satisfied, if both pairs
$(\underline{\partial} f_1(\overline{x}), - \overline{\partial} f_1(\overline{x}))$ and
$(\overline{\partial} f_1(\overline{x}), - \underline{\partial} f_1(\overline{x}))$ are in general position. 
Putting $v = 0$ in the definition of the general position one obtains that there exists 
$y^* \in \overline{\partial} f_1(\overline{x})$ such that $0 \notin \underline{\partial} f_1(\overline{x}) + y^*$ and
there exists $x^* \in \underline{\partial} f_1(\overline{x})$ such that 
$0 \notin x^* + \overline{\partial} f_1(\overline{x})$, that is, the constraint qualification from
Theorem~\ref{thrm:ContingConeToQuasidiffSet} is satisfied at $\overline{x}$ for some elements of the quasidifferential
of $f_1$ at $\overline{x}$. However, as in the case of inequality constraint, the opposite implication does not hold
true. In many particular cases the constraint qualification from Theorem~\ref{thrm:ContingConeToQuasidiffSet} is
satisfied for some elements of a quasidifferential of $f_1$ at $\overline{x}$, while one of 
the pairs $(\underline{\partial} f_1(\overline{x}), - \overline{\partial} f_1(\overline{x}))$ or
$(\overline{\partial} f_1(\overline{x}), - \underline{\partial} f_1(\overline{x}))$ is not in general position 
(take the function $f_1(x) = \max\{ 2 x^1, 2 x^2 \} + \min\{ 0, - x^1 - x^2 \}$ from
Example~\ref{example:CQ_ne_GeneralPosition}).

Optimality conditions for the more general problem
$$
  \min \: f_0(x) \quad \text{subject to} \quad F(x) = 0,
$$
where the mapping $F \colon X \to Y$ is \textit{scalarly quasidifferentiable} in a neighbourhood of $\overline{x}$,
similar to optimality condition \eqref{QuasidiffProg_OptCond_WithCone}, were studied by Uderzo \cite{Uderzo2,Uderzo2007}
in the case when a Banach space $Y$ admits a Fr\'{e}chet smooth renorming and a quasidifferential of $F$ satisfies
certain conditions \textit{in a neighbourhood of} $\overline{x}$ ensuring its metric regularity near this point. In
contrast, our conditions are formulated in terms of quasidifferentials at the point $\overline{x}$ itself and they can
be satisfied even if the equality constraints are not metrically regular near $\overline{x}$ (for example, the function
$f_1(x) = \max\{ \sin x^1 + \sin x^2, 0 \} + \min\{ - x^1 - x^2, x^1 \}$ from
Example~\ref{example:KernelDirectDeriv_ne_ContingCone} is not metrically regular near $\overline{x} = 0$).

Optimality conditions for the general problem
$$
  \min \: f_0(x) \quad \text{s.t.} \quad f_i(x) = 0 \quad \forall i \in I, \quad
  g_j(x) \le 0 \quad \forall j \in J,
$$
similar to \eqref{QuasidiffProg_LagrageMultipliers} were first obtained by Shapiro \cite{Shapiro84,Shapiro86} in the
case $X = \mathbb{R}^n$ under the assumption that for any $v \ne 0$ satisfying the equality
$f'_i(\overline{x}, v) = 0$ for all $i \in I$ the max-faces $\Delta(v \mid \underline{\partial} f_i(\overline{x}))$ and
$\Delta(v \mid \overline{\partial} f_i(\overline{x}))$ are singletones and the vectors $x_i^* - y_i^*$ with 
$\{ x_i^* \} = \Delta(v \mid \underline{\partial} f_i(\overline{x}))$ and
$\{ y_i^* \} = \Delta(v \mid \overline{\partial} f_i(\overline{x}))$, $i \in I$, are linearly independent. Observe that
this assumption is very hard to verify in nontrivial cases, since it requires the computation of the entire set
$\{ v \ne  0 \mid f'_i(\overline{x}, v) = 0 \: \forall i \in I \}$ and all corresponding max-faces. Furthermore, this
assumption is not satisfied in many particular cases. 

\begin{example}
Let $X = \mathbb{R}^2$ and $f_1(x) = \max\{ |x^2|, |x^2| - 2 x^1 \} + \min\{ x^1, 2 x^2 \}$. The
function $f_1$ is quasidifferentiable at the point $\overline{x} = 0$ and one can define 
$\mathscr{D} f_1(0) = [\underline{\partial} f_1(0), \overline{\partial} f_1(0)]$ with
$$
  \underline{\partial} f_1(0) = \co\left\{ \begin{pmatrix} 0 \\ 1 \end{pmatrix}, 
  \begin{pmatrix} 0 \\ -1 \end{pmatrix}, \begin{pmatrix} -2 \\ 1 \end{pmatrix},
  \begin{pmatrix} -2 \\ -1 \end{pmatrix} \right\} \quad
  \overline{\partial} f_1(0) = \co\left\{ \begin{pmatrix} 1 \\ 0 \end{pmatrix},
  \begin{pmatrix} 0 \\ 2 \end{pmatrix} \right\}.
$$
Note that for $v = (1, 0)^T$ one has $f'_1(\overline{x}, v) = 0$, but the max-face
$\Delta(v \mid \underline{\partial} f_i(\overline{x})) = \co\{ (0, \pm 1)^T \}$ is not a singleton, that is, the
constraint qualification from \cite{Shapiro84,Shapiro86} is not satisfied. On the other hand, 
for $y^* = (0, 2)^T \in \overline{\partial} f_1(\overline{x})$ 
one has $0 \notin \underline{\partial} f_1(\overline{x}) + y^*$ and 
for $x^* = 0 \in \underline{\partial} f_1(\overline{x})$ 
one has $0 \notin x^* + \overline{\partial} f_1(\overline{x})$, i.e. the constraint qualification 
from Theorem~\ref{thrm:ContingConeToQuasidiffSet} is satisfied
(see Corollary~\ref{crlr:ContingCone_EqualConstr}).
\end{example}

Finally, optimality conditions \eqref{QuasidiffProg_LagrageMultipliers} were first obtained by the author in
\cite{Dolgopolik_MetricReg} under significantly more restrictive assumptions than in
Corollary~\ref{crlr:QuasidiffProg_LagrangeMultipliers}. Namely, in \cite{Dolgopolik_MetricReg} it was assumed that the
functions $f_i$ and $g_j$ are (in some sense) semicontinuously quasidifferentiable in a neighbourhood of $\overline{x}$
and a weak q.d.-MFCQ holds at $\overline{x}$ (see Remark~\ref{rmrk:qdMFCQ}). As was noted above, the constraint
qualification that we use in this paper is much less restrictive than q.d.-MFCQ. Furthermore, in
Theorem~\ref{thrm:QuasidiffProgram_OptCond} and Corollary~\ref{crlr:QuasidiffProg_LagrangeMultipliers} we assume that
all functions are quasidifferentiable only at the point $\overline{x}$ and do not impose any assumptions on a
semicontinuity of the corresponding quasidifferential mappings.

\section{Conclusions}

In this paper, we presented a new description of convex subcones of the contingent cone to a set defined by
quasidifferentiable equality and inequality constraints. This description is based on the use of individual elements of
quasidifferentials of constraints and was inspired by the works of Di et al. \cite{Di1,Di2} on the derivation of
the classical KKT optimality conditions under weaker assumptions. Furthermore, the description of convex subcones
of the contingent cone provides one with a natural constraint qualification for nonsmooth mathematical programming
problems in terms of quasidifferentials and allows one to derive, apparently, the strongest quasidifferential-based
optimality conditions for such problems under the weakest possible assumptions. See \cite{Dolgopolik_ConstrainedCalcVar}
for applications of these constraint qualification and optimality conditions to constrained nonsmooth problems of the
calculus of variations.

The examples given in the end of the paper demonstrate that our constraint qualification can be satisfied for seemingly
degenerate problems, for which other constraint qualifications in terms of quasidifferentials fail. Furthermore, they
demonstrate that in some cases optimality conditions in terms of quasidifferentials are superior to the ones in terms of
various subdifferentials, since they are able to detect the nonoptimality of a given point, when optimality conditions
based on various subdifferentials fail to do so.

It should be noted that neither the description of convex subcones nor the constraint qualification and optimality
conditions presented in this paper are invariant under the choice of corresponding quasidifferentials. The invariance of
constraint qualifications, optimality conditions, descent directions etc. on the choice of quasidifferentials
has attracted a considerable attention of researchers (see, e.g. \cite{LudererRosiger91,Luderer92,WangMortensen});
however, it seems that non-invariant conditions depending on individual elements of quasidifferentials can lead to
stronger results.


\bibliographystyle{abbrv}  
\bibliography{OptimCond_bibl}

\end{document}